\documentclass[a4paper,leqno]{amsart}

\usepackage[english]{babel} 									
\usepackage[T1]{fontenc}							
\usepackage[utf8]{inputenx}
				
\usepackage{mathtools}								
\usepackage{empheq}									
\usepackage{enumitem}								

\usepackage{amssymb}								
\usepackage[sc]{mathpazo}							
\usepackage{mathrsfs}								

\usepackage{verbatim}                               
\usepackage{iflang}									
\usepackage{xifthen}								

\usepackage[final]{pdfpages}						
\usepackage{todonotes}								
\usepackage{aliascnt}								
\usepackage{needspace}								
\usepackage{refcount}

\usepackage{subcaption}			
\usepackage{placeins} 			
\usepackage{grffile}			
\usepackage[all]{xy}			

\usepackage{transparent}
\usepackage{color}				

\usepackage{graphicx}			

\usepackage[babel]{csquotes}							

\usepackage{nameref}									
\usepackage[colorlinks	=	true,						
          	 	linkcolor	=	red,
            	urlcolor	=	red,
            	citecolor	=	red]%
            	{hyperref}
\usepackage{bookmark}									
\usepackage{cleveref}									
\numberwithin{equation}{section}

\newtheorem{theorem}{Theorem}[section]

\newtheorem{corollary}[theorem]{Corollary}
\newtheorem{lemma}[theorem]{Lemma}
\newtheorem{proposition}[theorem]{Proposition}

\newtheorem{definition}[theorem]{Definition}

\theoremstyle{remark}
\newtheorem{remark}[theorem]{Remark}

\theoremstyle{remark}

\DeclareMathOperator{\lip}{lip}
\DeclareMathOperator{\LIP}{LIP}

\allowdisplaybreaks

\newcommand{\apmd}[2][]{							
	\ifthenelse{\equal{#1}{}}%
					{ \operatorname{N}_{#2}	}%
					{ \operatorname{N}_{#1,#2} 	}}

\newcommand{\aint}[2][]{
	\ifthenelse{\equal{#1}{}}%
					{%
\mathchoice%
      {\mathop{\kern 0.2em\vrule width 0.6em height 0.69678ex depth -0.58065ex
              \kern -0.8em \intop}\nolimits_{\kern -0.45em#2}^{#1}}%
      {\mathop{\kern 0.1em\vrule width 0.5em height 0.69678ex depth -0.60387ex
              \kern -0.6em \intop}\nolimits_{#2}^{#1}}%
      {\mathop{\kern 0.1em\vrule width 0.5em height 0.69678ex depth -0.60387ex
              \kern -0.6em \intop}\nolimits_{#2}^{#1}}%
      {\mathop{\kern 0.1em\vrule width 0.5em height 0.69678ex depth -0.60387ex
              \kern -0.6em \intop}\nolimits_{#2}^{#1}}}%
					{%
\mathchoice%
      {\mathop{\kern 0.2em\vrule width 0.6em height 0.69678ex depth -0.58065ex
              \kern -0.8em \intop}\nolimits_{\kern -0.45em#1}^{#2}}%
      {\mathop{\kern 0.1em\vrule width 0.5em height 0.69678ex depth -0.60387ex
              \kern -0.6em \intop}\nolimits_{#1}^{#2}}%
      {\mathop{\kern 0.1em\vrule width 0.5em height 0.69678ex depth -0.60387ex
              \kern -0.6em \intop}\nolimits_{#1}^{#2}}%
      {\mathop{\kern 0.1em\vrule width 0.5em height 0.69678ex depth -0.60387ex
              \kern -0.6em \intop}\nolimits_{#1}^{#2}}}}

\begin{document}
\title[Pushforward of currents]{Pushforward of currents under Sobolev maps}

\author{Toni Ikonen} 
\address{Department of Mathematics and Statistics, P. O. Box 68 (Pietari Kalmin katu 5), FI-00014 University of Helsinki, Finland}

\email{toni.ikonen@helsinki.fi}

\keywords{calibration, cochain, integral current, isoperimetric inequality, pushforward, quasiregular, Sobolev}
\thanks{This work was supported by the Academy of Finland, project number 332671. \newline {\it 2020 Mathematics Subject Classification.} Primary: 46E36. Secondary: 30C65, 49Q15, 53C65}


\begin{abstract}
We prove that a Sobolev map from a Riemannian manifold into a complete metric space pushes forward almost every compactly supported integral current to an Ambrosio--Kirchheim integral current in the metric target, where "almost every" is understood in a modulus sense.

As an application, we prove that when the target supports an isoperimetric inequality of Euclidean type for integral currents, an isoperimetric inequality for Sobolev mappings relative to bounded, closed and additive cochains follows. 

Using the results above, we answer positively to an open question by Onninen and Pankka on sharp Hölder continuity for quasiregular curves. A key tool in the continuity proof is Almgren's isoperimetric inequality for integral currents.
\end{abstract}

\maketitle\thispagestyle{empty}

\section{Introduction}\label{sec:intro}

\subsection{Pushforward and Sobolev maps}
Given Riemannian manifolds $M$ and $N$, Federer--Fleming integral currents $\mathbf{I}_{k,c}( M )$ \cite{Fed:Fle:60}, for every $k = 0,1,\dots$, can be pushed forward using smooth or Lipschitz mappings $f \colon M \to N$. We prove a Sobolev analog of the statement. We use a notion of an exceptional collection of integral currents in the spirit of Fuglede \cite{Fug:57}. To this end, given a Borel function $\rho_0 \in \mathcal{L}^{p}_+( M )$, we say that
\begin{align*}
    \mathcal{F}( \rho_0 )
    \coloneqq
    \bigcup_{ k \in \left\{ 1, \dots, \lfloor p \rfloor \right\} }
    &\left\{ 
        T\in \mathbf{I}_{k,c}( M )
        \colon
        \int_{M} \rho_0^{ k } \,d\|T\|
        +
        \int_{M} \rho_{0}^{ k-1 } \,d\| \partial T \|
        =
        \infty
    \right\}
    \\
    \cup
    &\left\{ T \in \mathbf{I}_{0,c}( M ) \colon \int_{M} \rho_0^{0} \,d\|T\| = \infty \right\}
\end{align*}
is \emph{$p$-exceptional}, where the conventions $0^{0} = 0$ and $\infty^0 = \infty$ are used.

We consider the \emph{reduced chain complex $( \mathbf{I}_{\star, c}( \rho_0 ), \partial )$} consisting of the direct sum of the $\mathbb{Z}$-modules
\begin{equation*}
    \mathbf{I}_{k,c}( \rho_0 )
    =
    \mathbf{I}_{k,c}( M ) \setminus \mathcal{F}( \rho_0 ),
    \quad
    k \in \left\{0,1,\dots,\lfloor p \rfloor\right\}
\end{equation*}
and the usual boundary operation of currents, denoted $\partial$.

Our main result, \Cref{thm:pushforward:intro} below, states the existence of a chain morphism
\begin{equation*}
    f_\star \colon ( \mathbf{I}_{\star, c}( \rho_0 ), \partial )
    \to
    ( \mathbf{I}_\star( N ), \partial ),
\end{equation*}
induced by a Sobolev map $f \in W^{1,p}( M, N )$ from $( \mathbf{I}_{\star, c}( \rho_0 ), \partial )$ into the chain complex $( \mathbf{I}_{\star}( N ), \partial )$ of integral currents in the sense of Ambrosio--Kirchheim \cite{AK:00:current}.

Proving the existence of the pushforward involves approximating Sobolev mappings by Lipschitz mappings and for this reason we first need to isometrically embed $N$ into an injective Banach space, by using a Kuratowski embedding, for example. The embedding is used since there are topological obstructions for $N$-valued Lipschitz mappings being dense in the Sobolev spaces of interest, cf. \cite{Bet:91,Hang:Lin:03,Haj:09,Haj:Sch:14,Bou:Pon:VanSc:17}. The core of our argument involves Sobolev theory tailored for metric spaces and for mappings taking values in injective Banach spaces. In fact, replacing $N$ by a complete Lipschitz manifold, or any complete metric space for that matter, poses no further challenge.
\begin{theorem}\label{thm:pushforward:intro}
Suppose that $M$ is a Riemannian manifold, with or without boundary, and $N$ is a complete metric space and $f \in W^{1,p}_{loc}( M, N )$ for $1 \leq p < \infty$.

Then there exist a Borel representative $( \widehat{f}, \widehat{\rho} )$ of $( f, |df| )$ and Borel $\widehat{\rho}_0 \in \mathcal{L}^{p}_+( M )$ together with a chain morphism
\begin{equation}\label{eq:pushforward:intro}
    \widehat{f}_\star \colon ( \mathbf{I}_{\star, c}( \widehat{\rho}_0 ), \partial ) \to ( \mathbf{I}_{\star}( N ), \partial )
\end{equation}
for which
\begin{equation}\label{eq:mass-energy:intro}
    \| \widehat{f}_\star T \|( E )
    \leq
    \int_{ \widehat{f}^{-1}(E) } \widehat{\rho}^k \,d\|T\|
    <
    \infty
    \quad\text{for Borel $E \subset N$},
\end{equation}
whenever $T \in \mathbf{I}_{k,c}( \widehat{\rho}_0 )$ and $k \in \left\{ 0,1,\dots, \lfloor p \rfloor \right\}$.

The chain morphism has the further property that for every isometric embedding $\iota \colon N \to \mathbb{V}$ to 
 an injective Banach space $\mathbb{V}$, there is a Borel decomposition sets $( E_i )_{ i = 0 }^{ \infty }$ of $M$, with negligible $E_0 \supset \left\{ \rho_0 = \infty \right\}$, and Lipschitz mappings $f_i \colon M \to \mathbb{V}$ extending $( \iota \circ f )|_{ E_i }$ for $i \geq 1$, for which
\begin{equation}\label{eq:pointwiseconvegence}
    \iota_\sharp\left( \widehat{f}_\star T \right)
    =
    \sum_{ i = 1 }^{ \infty }
        ( f_i )_\sharp(  T\llcorner E_i )
    \quad\text{for every $T \in \mathbf{I}_{\star,c}( \widehat{\rho}_0 )$.}
\end{equation}
\end{theorem}
When applying \Cref{thm:pushforward:intro}, we identify $\widehat{f}$ with $f$ and $\widehat{\rho}$ with the $p$-minimal $p$-weak upper gradient $|df|$, with the understanding that \Cref{thm:pushforward:intro} is applied using the representatives $( \widehat{f}, \widehat{\rho} )$ as above.

The notation $h_\sharp$ for Lipschitz $h$ refers to the usual pushforward along Lipschitz mappings, cf. \Cref{sec:current:cochain}, and equality \eqref{eq:pointwiseconvegence} is intended as a pointwise convergence of the sequence of currents $\sum_{ i = 1 }^{ l } ( f_i )_\sharp(  T\llcorner E_i )$, $l \in \mathbb{N}$, to $\iota_\sharp\left( \widehat{f}_\star T \right)$. This technical condition, and standard locality properties of currents, guarantees that for any pair of chain maps satisfying the conclusions of \Cref{thm:pushforward:intro}, there is $\rho_0 \in \mathcal{L}^p_{+}( M )$ such that the pair of chain maps are well-defined and equal in $( \mathbf{I}_{\star, c}( \rho_0 ), \partial )$. Observe that the set $E_0$ plays no role in \eqref{eq:pointwiseconvegence}.

The statement above uses the integral currents of Ambrosio and Kirchheim, while the Sobolev theory follows the upper gradient approach due to Heinonen and Koskela \cite{Hei:Kos:98,Sha:00,Hei:Kos:Sha:Ty:01}. We recall that in Riemannian manifolds, the classes of integral currents in the sense of Federer--Fleming and compactly supported integral currents in the sense of Ambrosio--Kirchheim coincide, cf. \cite[Theorem 11.1]{AK:00:current}, as do the upper gradient and Reshetnyak approaches to metric-valued Sobolev theory, cf. \cite{Res:97,Sha:00,Hei:Kos:Sha:Ty:01,Hei:Kos:Sha:Ty:15}.

\subsection{Pullback of cochains}
It is well-understood that Sobolev maps between Riemannian manifolds induce a \emph{pullback} $f^{\star}$ of bounded differential forms, defined using the weak differential of Sobolev mappings.

The pullback $f^{\star}$ induces a chain morphism from the (restricted) de Rham complex of bounded and smooth differential forms to the de Rham complex of Sobolev differential forms; various authors have considered cohomology theories based on the Sobolev complexes, see, for example, \cite{Don:Sul:89,Iwa:Mar:93,Pan:08,Haj:Iwa:Mal:On,Gol:Tro:10,Kan:Pan:19}. \Cref{thm:pushforward:intro} leads to a similar theory on the homology level.

We establish a pullback version of \Cref{thm:pushforward:intro} in the metric setting, partially motivated by the Riemannian results above. We formulate the pullback using the metric cochains due to Rajala and Wenger \cite{Raj:Wen:13}.

Given a metric space $N$ and $k \in \mathbb{N}$, we consider a $\mathbb{Z}$-module $\mathcal{C} \subset \mathbf{I}_{k}(N)$. We say that
\begin{equation}\label{eq:cochain}
    \Omega \colon \mathcal{C} \to [-\infty, \infty]
\end{equation}
is a \emph{subadditive cochain} if $\Omega(0) = 0$ and
\begin{equation*}
    | \Omega( T ) |
    \leq
    | \Omega( T + S ) |
    +
    | \Omega( S ) |
    \quad\text{for every $T,S \in \mathcal{C}$.}
\end{equation*}
If also $\Omega( T + S ) = \Omega( T ) + \Omega( S )$ whenever each term is finite, we say that $\Omega$ is \emph{additive}. In case $\Omega( T )$ is always finite, we say that $\Omega$ is \emph{finite}.

A subadditive $\Omega$ has an \emph{upper norm} $\rho \colon N \to [0,\infty]$ if $\rho$ is Borel and
\begin{equation*}
    \quad
    | \Omega( T ) |
    \leq
    \int_{ N }
        \rho
    \,d\|T\|
    \quad\text{for all $T \in \mathcal{C}$.}
\end{equation*}
It is crucial for us later, see \Cref{thm:yes} below, that the mass measure $\|T\|$ is defined in the Ambrosio--Kirchheim sense. Following \cite{Raj:Wen:13}, we denote by $\mathcal{L}_{\infty}( \mathcal{C} )$ the collection of all \emph{subadditive} cochains with a bounded upper norm. 

Using the notation from \Cref{thm:pushforward:intro}, we set
\begin{equation}\label{eq:pullbackdefinition}
    f^{\star}\Omega( T )
    \coloneqq
    \Omega( f_\star T ),
    \quad
    ( T, \Omega ) \in \mathbf{I}_{k,c}( \widehat{\rho}_0 ) \times \mathcal{L}_{\infty}( \mathbf{I}_{k}(N) ), \quad k \in \left\{0, \dots, \lfloor p \rfloor \right\}.
\end{equation}
By \Cref{thm:pushforward:intro}, we obtain the following.
\begin{corollary}[Pullback]\label{thm:pullback}
Let $M$ be a Riemannian manifold, with or without boundary, $N$ a complete metric space and $f \in W^{1,p}_{loc}( M, N )$, $1 \leq p < \infty$. If $k \in \left\{ 0,1,\dots, \lfloor p \rfloor \right\}$ and $\Omega \in \mathcal{L}_{\infty}( \mathbf{I}_{k}(N) )$ has a bounded upper norm $\rho$, the finite subadditive cochain
\begin{equation*}
    f^{\star}\Omega \colon \mathbf{I}_{k,c}( \widehat{\rho}_0 ) \to \mathbb{R}
\end{equation*}
has an upper norm $( \rho \circ f )|df|^{k} \in \mathcal{L}^{p/k}_{loc}( M )$. In case $\Omega$ is additive, so is the pullback.
\end{corollary}
We use the convention $p/0 = \infty$ and recall the notation
\begin{equation*}
    |df|^{0}
    =
    \left\{
    \begin{split}
        &0, \quad&&\text{when $|df| = 0$,}
        \\
        &\infty, \quad&&\text{when $|df| = \infty$,}
        \\
        &1, \quad&&\text{otherwise}.
    \end{split}
    \right.
\end{equation*}
We refine the conclusion of \Cref{thm:pullback} in the case that the \emph{coboundary} of $\Omega \in \mathcal{L}_{ \infty }( \mathbf{I}_{k}(N) )$ is bounded. To this end, given a $\mathbb{Z}$-module $\mathcal{C}' \subset \mathbf{I}_{k+1}( N )$ with $\partial \mathcal{C}' \subset \mathcal{C}$, we call the map
\begin{equation*}
    \delta \Omega \colon \mathcal{C}' \to \mathbb{R},
    \quad S \mapsto \Omega( \partial S ),
\end{equation*}
the \emph{coboundary} of $\Omega$ (relative to $\mathcal{C}'$). We denote $\Omega \in \mathcal{W}_{\infty,\infty}( \mathcal{C}, \mathcal{C}' )$ in case $\Omega$ and its coboundary have a bounded upper norm. We say that $\Omega$ is \emph{closed} (relative to $\mathcal{C}'$) if the coboundary is identically zero.

When $N$ is a Riemannian manifold, we denote $\Omega \in \mathcal{W}_{\alpha,\beta,loc}( \mathcal{C}, \mathcal{C}' )$ in case $\Omega$ has a locally $\alpha$-integrable upper norm, $\partial \mathcal{C}' \subset \mathcal{C}$, and $\delta \Omega$ has a locally $\beta$-integrable upper norm. \Cref{thm:pullback} implies the following.
\begin{corollary}[Sobolev cochains]\label{thm:pullback:Sobolev}
Let $M$ be a Riemannian manifold, with or without boundary, $N$ a complete metric space and $f \in W^{1,p}_{loc}( M, N )$, $1 \leq p < \infty$. If $k \in \left\{ 0,1,\dots, \lfloor p \rfloor \right\}$ and $k \leq \mathrm{dim}(M)$, there is an $\mathbb{R}$-linear pullback
\begin{equation*}
    f^{\star}
    \colon
    \mathcal{W}_{\infty,\infty}( \mathbf{I}_{k}( N ), \mathbf{I}_{k+1}(N ) )
    \to
    \mathcal{W}_{p/k, p/(k+1), loc}( \mathbf{I}_{k,c}( \widehat{\rho}_0 ), \mathbf{I}_{k+1,c}( \widehat{\rho}_0 ) ).
\end{equation*}
In case $k = \mathrm{dim}(M)$, the pullback is always closed.
\end{corollary}
Given any $k$ as in \Cref{thm:pullback:Sobolev}, any element $\Omega$ in the codomain of $f^{\star}$ admits a subadditive extension to
\begin{equation}\label{eq:sobolevcochain}
    \mathcal{W}_{p/k, p/(k+1), loc}( \mathbf{I}_{k,c}( M ), \mathbf{I}_{k+1,c}( M ) ),
\end{equation}
simply by extending $\Omega$ and its coboundary as infinity to the set where they are not yet defined. The spaces in \eqref{eq:sobolevcochain} for $k \in \left\{0,1\dots,\lfloor p \rfloor \right\}$ are the collection of (local) Sobolev cochains in the sense of \cite[Section 3.2]{Raj:Wen:13}. When $p = \mathrm{dim}(M)$, the direct sum of the spaces over $k \in \left\{0, 1, \dots, \mathrm{dim}(M) \right\}$, endowed with the coboundary operation, provides a reduced cochain complex that is a \emph{quasiconformal invariant} of $M$. This is readily seen from \Cref{thm:pushforward:intro} and \Cref{thm:pullback}.

When $M = \mathbb{R}^n$, $p \in [n,\infty)$, $k \in \left\{1, \dots, n\right\}$, and $\Omega$ is an additive element from the space in \eqref{eq:sobolevcochain}, the action of $\Omega$ on polyhedral chains is uniquely determined by a Sobolev differential form $\tau \in \mathcal{W}_{d, loc }^{ p/k, p/(k+1) }( \Lambda^{k} \mathbb{R}^n )$; this follows from the Wolfe-type representation result \cite[Theorem 1.1]{Cam:Raj:Wen:15}. This provides a concrete representation of the pullback in this special case. In fact, if $f$ is Lipschitz, the form can be chosen to be a flat Whitney form, see \cite[Corollary 1.2]{Cam:Raj:Wen:15}.

\subsection{Closed cochains and an isoperimetric inequality for Sobolev maps}
\Cref{thm:pushforward:intro} immediately implies the following isoperimetric inequality for Sobolev mappings.
\begin{corollary}\label{cor:Sob:isoperimetry}
Let $n \geq 2$. Suppose that $M$ is an oriented $n$-dimensional Riemannian manifold, with or without boundary, $f \in W^{1,n}_{loc}( M, N )$, and $\Omega \in \mathcal{L}_\infty( \mathbf{I}_{n}(N) )$ is a closed and additive cochain with a bounded upper norm $\rho(x) = C$, $x \in N$, for some $C > 0$.

If $N$ supports isoperimetric inequalities of Euclidean type with dimension $n$ and constant $A'$, and dimension $n-1$ and constant $A$ in the sense of \eqref{eq:smallmass}, respectively, then
\begin{align}\label{eq:closedcochain:inequality}
    | \Omega( f_\star[\widetilde{M}] ) |
    &\leq
    C
    \inf_{ S \in \mathbf{I}_{n+1}(N) }
    M( f_\star[\widetilde{M}] + \partial S )
    \leq
    A C
    \left( \int_{ \partial \widetilde{M} } |df|^{n-1} \,d\mathcal{H}^{n-1} \right)^{ \frac{n}{n-1} }
\end{align}
whenever $[ \widetilde{M} ] \in \mathbf{I}_{n,c}( \widehat{\rho}_0 )$.
\end{corollary}
We consider top-dimensional compact Lipschitz submanifolds $\widetilde{M}$ of $M$ as integral currents~$[\widetilde{M}]$, with orientation induced from $M$.

When deriving \Cref{cor:Sob:isoperimetry}, the $n$-dimensional isoperimetric inequality is only applied to deduce that every integral $n$-cycle is the boundary of some $\mathbf{I}_{n+1}(N)$, which, together with the closedness of $\Omega$, allows us to bound the left-hand side of \eqref{eq:closedcochain:inequality} by the minimal mass in the homology class of $f_\star [\widetilde{M}]$. The $(n-1)$-dimensional inequality gives the quantitative upper bound in \eqref{eq:closedcochain:inequality}.

\Cref{cor:Sob:isoperimetry} can be considered a generalization of the classical isoperimetric inequality for Sobolev mappings proved by Reshetnyak, cf. \cite{Res:66:isoperi} or \cite[Lemma 1.2, p. 80]{Res:89}. Indeed, \cite[Lemma 1.2]{Res:89} follows from \eqref{eq:closedcochain:inequality} and Almgren's sharp isoperimetric inequality \cite[Theorem 10]{Alm:86} when $N = \mathbb{R}^n$ and $\Omega$ is the cochain induced by the Riemannian volume form of $N$. \Cref{cor:Sob:isoperimetry} plays a key role in our main application, \Cref{thm:yes} below.

The isoperimetric inequality for integral currents on Euclidean spaces was initially proved by Federer and Fleming \cite{Fed:Fle:60}, the sharp isoperimetry constant obtained by Almgren \cite{Alm:86}. These results were later extended to the Banach and metric settings by Gromov and Wenger~\cite{Gro:83,Wen:07}.

\subsection{Quasiregular curves}
We consider a complete Riemannian manifold $N$ and a closed and bounded differential form $\omega \in \Omega^{n}( N )$ for some $n \leq \mathrm{dim}( N )$.

Given an oriented Riemannian $n$-manifold $M$, we say that $f \colon M \to N$ is a \emph{$K$-quasiregular $\omega$-curve} if $f \in W^{1,n}_{loc}( M, N )$ and
\begin{equation}\label{eq:distortioninequality}
    ( \|\omega\| \circ f ) |df|^n
    \leq
    K\,
    {\star}f^{\star}\omega
    \quad\text{$\mathcal{H}^n$-almost everywhere,}
\end{equation}
where
\begin{equation}\label{eq:comass}
    \|\omega\|(x)
    =
    \sup\left\{
        \omega( v )
        \colon
        \text{$v$ is a unit simple $n$-vector on $T_xN$}
    \right\}
\end{equation}
is the \emph{comass} of $\omega$, and ${\star}f^\star \omega$ is the Hodge star of the pullback $f^{\star} \omega$, defined pointwise using the weak differential of $f$.

In contrast to the definition of quasiregular curves in \cite{Pan:20}, we do not assume \emph{a priori} that $f$ is \emph{continuous} or that $\omega$ is \emph{strongly nonvanishing} (i.e. $\|\omega\|(x) \neq 0$ for every $x \in N$). In fact, in our applications, we only assume that $\| \omega \|(x)$ is bounded from above by some $C > 0$ everywhere and from below by $c > 0$ in the image of $f$. These assumptions are motivated by the calibrated geometries of Harvey and Lawson \cite{Har:Law:82}.

There are various contexts at which quasiregular curves naturally occur. Standard applications include pseudoholomorphic curves on Kähler manifolds, cf. \cite{Grom:85,Don:96,Gro:07}. Further examples can be obtained from the classical theory of quasiregular mappings \cite{Res:89,Rick:93,Iw:Ma:01,Ast:Iwa:Mar:09}. See \cite{Smi:11,Che:Kari:Mad:20,Hei:21,He:Pa:Pry:23} for related work.

Reshetnyak proved that every $K$-quasiregular mapping from $U \subset \mathbb{R}^n$ into $\mathbb{R}^n$ is locally $(1/K)$-Hölder continuous, cf. \cite{Res:66,Res:89}. This was later extended to $K$-quasiregular $\omega$-curves for nonzero constant-coefficient $n$-forms $\omega$ by Onninen and Pankka in \cite[Theorem 1.1]{Onn:Pan:21} with a smaller Hölder exponent. We prove the following sharp version of \cite[Theorem 1.1]{Onn:Pan:21}, answering positively to a question posed by Onninen and Pankka.
\begin{theorem}\label{thm:yes}
Every $K$-quasiregular $\omega$-curve $f \colon U \to \mathbb{R}^m$ for open $U \subset \mathbb{R}^n$ and nonzero constant-coefficient $n$-form $\omega$ has a locally $(1/K)$-Hölder continuous representative.
\end{theorem}
We outline the proof of \Cref{thm:yes}. First, the constant-coefficient form $\omega$ in \Cref{thm:yes} induces a closed and additive cochain $\Omega \in \mathcal{W}_{\infty, \infty}( \mathbf{I}_n( \mathbb{R}^m ), \mathbf{I}_{n+1}( \mathbb{R}^m ) )$. In fact, the definitions of mass due to Ambrosio and Kirchheim and of upper norms for cochains allow us to identify the comass $x \mapsto \rho(x) = \|\omega\|(x)$ from \eqref{eq:comass} as an upper norm of $\Omega$ after which we may apply \eqref{eq:closedcochain:inequality} and \eqref{eq:distortioninequality}; the results are applicable because the pullback $f^{\star}\Omega$ coincides with the cochain induced by the pullback $f^{\star}\omega$, cf. \Cref{lemm:consistency:Pankka}. These observations and the sharp isoperimetric inequality \cite[Theorem 10]{Alm:86} reduce \Cref{thm:yes} to a standard argument due to Morrey. We also prove the following.
\begin{corollary}\label{cor:regularity:lip+harm}
Every $1$-quasiregular $\omega$-curve $f \colon U \to \mathbb{R}^m$ for open $U \subset \mathbb{R}^n$ and nonzero constant-coefficient n-form $\omega$ has a representative that is locally Lipschitz and $n$-harmonic.
\end{corollary}
A simple argument, using the closedness of the cochain induced by $\omega$, shows that if $f$ is as in \Cref{cor:regularity:lip+harm}, the pushforward integral currents $f_{\star}[ \widetilde{M} ]$, for compact and oriented Lipschitz submanifolds $\widetilde{M}$, are integral currents that are calibrated in the sense of Harvey and Lawson, cf. \cite[Section 4]{Har:Law:82}. In particular, such pushforward integral currents are mass-minimizing in their homology class. A similar idea is used to prove the $n$-harmonicity of the continuous representative of $f$.

\subsection{Further observations}
The proof of \Cref{thm:yes}, based on \Cref{thm:pushforward:intro} and Corollaries \ref{thm:pullback} and \ref{cor:Sob:isoperimetry}, allows us to generalize the modulus of continuity estimate to various settings. Indeed, the constant-coefficient $n$-forms in \Cref{thm:yes} can be replaced by closed flat Whitney forms satisfying $\|\omega\|(x) \geq c$ pointwise on the image of $f$, in which case $K$ in the modulus of continuity estimate ought to be replaced by $Q = K\|\omega\|_{\infty}/c$, where $\|\omega\|_\infty$ is the $\mathcal{H}^n$-essential supremum of the comass of $\omega$. We later formulate a similar extension to the class of closed and bounded additive cochains. 

A quantitative local Hölder continuity as in \Cref{thm:yes} can be obtained in various contexts, e.g. on any compact Lipschitz manifold $N$, or more generally, complete metric targets supporting isoperimetric inequalities of Euclidean type for the dimensions $n-1$ and $n$ up to some mass scale; see \Cref{sec:singularsetting} for the precise definition. Targets satisfying these assumptions include complete Riemannian manifolds with an upper bound on sectional curvature. Subtler examples are obtained by considering CAT($\kappa$)-spaces (see e.g. \cite{Wen:07}) or Heisenberg group of high enough dimension, endowed with any standard left-invariant subRiemannian distance, e.g. the Korányi distance, cf.~\cite{Bas:Wen:You:21}.

As far as the author is aware, the (local) Hölder continuity is new already for quasiregular mappings taking values in complete Riemannian manifolds with a sectional curvature upper bound, see \cite{Gol:Haj:Pak:19} and the recent monograph \cite{Kan:21} for further discussion.

We elaborate on these topics in \Cref{sec:singularsetting}.

\subsection{Organization of the paper}
In \Cref{sec:preliminaries}, we introduce our notations and basic definitions and recall some facts regarding injective Banach spaces. \Cref{sec:integral} consists of preliminary work needed for the proof of \Cref{thm:pushforward:intro}, while \Cref{thm:pushforward:intro} and \Cref{thm:pullback} are proved in \Cref{sec:mainresults:proof}.

We recall a representation theorem for integer-rectifiable currents in \Cref{sec:earlier} which illustrates the connection of comass of differential forms to the upper norm of the induced additive cochains. \Cref{cor:Sob:isoperimetry}, \Cref{thm:yes} and \Cref{cor:regularity:lip+harm} are proved in \Cref{sec:singularsetting}. That section also contains a generalization of quasiregular curves based on closed additive cochains.

\section{Preliminaries}\label{sec:preliminaries}
We introduce the notation and basic definitions we need during the manuscript.

\subsection{Lipschitz mappings}
Given metric spaces $N$ and $Y$, $\LIP( N, Y )$ refers to the Lipschitz mappings $h \colon N \rightarrow Y$. We say that $h$ is Lipschitz whenever
\begin{equation*}
    \LIP( h )
    \coloneqq
    \inf\left\{ 
        L > 0
        \colon
        d(h(x),h(y)) \leq L d(x,y) \,
        \text{for every $(x,y) \in N \times N$}
    \right\}
    <
    \infty.
\end{equation*}
When $\LIP( h ) \leq L$, we say that $h$ is $L$-Lipschitz.

The class $\LIP_{b}( N, Y )$ refers to the collection of all $h \in \LIP( N, Y )$ whose image is a bounded subset of the image. When $Y$ is a Banach space and $f \in \LIP( N, Y )$, the set $\overline{ \left\{ f \neq 0 \right\} }$ is the \emph{support} of $f$. The class $\LIP_{bs}( N, Y )$ consists of all Lipschitz mappings with \emph{bounded support}, i.e. mappings whose support has finite diameter. In particular, $\LIP_{bs}( N, Y ) \subset \LIP_{b}( N, Y )$. When $Y = \mathbb{R}$ we omit $Y$ from the notations above.

Given any metric spaces $N$ and $Y$, $f \in \LIP(N, Y)$ and $x \in N$, we denote
\begin{equation*}
    \lip(f)(x)
    \coloneqq
    \limsup_{ x \neq y \rightarrow x}
    \frac{ d(f(x),f(y)) }{ d(x,y) },
\end{equation*}
and call $\lip(f)(x)$ the \emph{pointwise Lipschitz constant} of $f$ at $x$. The corresponding function is called the \emph{pointwise Lipschitz function} of $f$.

\subsection{Injective Banach spaces}
Given a nonempty index set $N$, we consider the space
\begin{equation*}
    L^{\infty}( N )
    \coloneqq
    \left\{
        a \colon N \to \mathbb{R}
        \colon
        \sup_{ x \in N } |a(x)| < \infty
    \right\}
\end{equation*}
endowed with componentwise addition and summation together with the supremum norm $\| a \| = \sup_{ x \in N } |a(x)|$. Then $L^{\infty}(N)$ becomes a Banach space. We typically denote $( a_x )_{ x \in N }$ instead of $a \colon N \to \mathbb{R}$.

When $N$ is a metric space, we use the \emph{Kuratowski embedding} $\iota \colon N \to L^{\infty}(N)$ to isometrically embed a metric space $N$ into $L^{\infty}(N)$: fix $x_0 \in N$, consider $f_{y}(x) \coloneqq d( x, y ) - d( y, x_0 )$, and set $\iota(x) = ( f_{y}(x) )_{ y \in N }$. We often identify $N$ with its image $\iota(N)$. Observe that if $N' \subset N$ is a dense subset, we may define an isometric embedding of $N$ into $L^{\infty}(N')$ using the formula $\iota'(x) = ( f_y(x) )_{ y \in N' }$ as above. In particular, when $N$ is separable, this idea yields an isometric embedding into $\ell^{\infty} \coloneqq L^{\infty}( \mathbb{N} )$.

A Banach space $\mathbb{V}$ is \emph{injective} if whenever $\mathbb{W}$ and $\mathbb{H}$ are Banach spaces, $\iota \colon \mathbb{W} \to \mathbb{H}$ is a linear isometric embedding, and there is a bounded linear map $L \colon \mathbb{W} \to \mathbb{V}$, then there is a linear map $\widetilde{L} \colon \mathbb{H} \to \mathbb{V}$ with $\widetilde{L} \circ \iota = L$, so that the operator norms of $L$ and $\widetilde{L}$ coincide.
\begin{lemma}\label{lem:inj:infty}
Given any nonempty index set $N$, the space $L^{\infty}(N)$ is injective. Furthermore, if $\mathbb{V}$ is injective and $N$ is the closed unit ball of $\mathbb{V}^{*}$, there is a linear isometric embedding $\iota \colon \mathbb{V} \to L^{\infty}( N )$ and a linear $P \colon L^{\infty}( N ) \to \mathbb{V}$ so that $P$ has unit operator norm and $P \circ \iota$ is the identity.
\end{lemma}
\begin{proof}
Given an index set $N$, we endow it with the topology making every set open. Then $L^{\infty}(N)$ can be linearly isometrically isomorphically identified with the space of continuous functions $C( \beta N )$ for the Stone--\v{C}ech compactification $\beta N$ of $N$. Now Hahn--Banach theorem implies the claim for $C( \beta N )$, see \cite[Chapter 3, Section 11, Theorem 3, p. 87]{Lac:74}. Thus $L^{\infty}(N)$ is injective.

In case $\mathbb{V}$ is injective, we denote by $N$ the closed unit ball of the dual $\mathbb{V}^{*}$ and consider the linear isometric embedding $\iota(v) = ( z(v) )_{ z \in N }$. The existence of $P$ follows by extending $L = \iota^{-1}$ from $\mathbb{W} = \iota( \mathbb{V} )$ using the definition of injectivity of $\mathbb{V}$.
\end{proof}

On the one hand, \Cref{lem:inj:infty} shows that $L^{\infty}( N )$ for arbitrary index sets $N \neq \emptyset$ are injective. On the other hand, \Cref{lem:inj:infty} implies the following key lemma.
\begin{lemma}[McShane for $L$-Lipschitz mappings]\label{lemm:mcshane}
Let $\mathbb{V}$ be an injective Banach space, $K \subset M$ a set in a metric space $M$, and $L \geq 0$. Let $\LIP_{b,L}( K, \mathbb{V} ) \subset \LIP_{b}( K, \mathbb{V} )$ denote the collection of $L$-Lipschitz mappings. Then there is an extension operator
\begin{equation*}
    \mathcal{E}
    \colon
    \LIP_{b,L}( K, \mathbb{V} ) \to \LIP_{b}( M, \mathbb{V} )
\end{equation*}
satisfying the following:
\begin{itemize}
    \item for each $F \in \LIP_{b,L}( K, \mathbb{V} )$, the extension $\mathcal{E}( F )$ is $L$-Lipschitz;
    \item If $F, G \in \LIP_{b,L}( K, \mathbb{V} )$, then $$\sup_{ x \in K } | F(x) - G(x) | = \sup_{ x \in M } | \mathcal{E}( F )(x) - \mathcal{E}( G )(x) |.$$
\end{itemize}
\end{lemma}
\begin{proof}
\Cref{lem:inj:infty} implies that it suffices to prove the claim when $\mathbb{V} = L^{\infty}(N)$, for arbitrary $N \neq \emptyset$. To verify the special case, consider an $L$-Lipschitz $F \in \LIP_{b}( K, \mathbb{V} )$ and identify $F$ with $( f_z )_{ z \in N }$ for $f_z \colon K \to \mathbb{R}$.

Every $f_z$ is $L$-Lipschitz and $|f_z(x)| \leq |F(x)|$ for all $(x,z) \in K \times N$, so we may consider the McShane extension
\begin{equation*}
    \mathcal{E}( f_z )(x)
    \coloneqq
    \inf_{ y \in K }\left\{ f_z(y) + Ld( y, x ) \right\}, \quad x \in M.
\end{equation*}
Then $\mathcal{E}( f_z )$ extends $f_z$, is $L$-Lipschitz, and if $g \in \LIP_{b,L}( K )$, we have
\begin{equation}\label{eq:monotonicity:extension}
    \sup_{ x \in M }
    | \mathcal{E}( f_z )(x) - \mathcal{E}( g )(x) |
    =
    \sup_{ y \in K }
    | f_z(y) - g(y) |.
\end{equation}
We define $\mathcal{E}( F )(x)
    \coloneqq
    \left( \mathcal{E}( f_z )(x) \right)_{ z \in N }$ for every $x \in M$,
observing that $\mathcal{E}( F )$ is an $L$-Lipschitz extension of $F$. Then \eqref{eq:monotonicity:extension} implies the claim.
\end{proof}

\subsection{Measure theory}
Given a metric space $M$ endowed with a(n outer) measure $\mu$, a set $E \subset M$ is \emph{$\mu$-negligible} if $\mu( E ) = 0$. We say that a property holds \emph{almost everywhere} if the property is true for every point $x \in M \setminus E$ for a $\mu$-negligible set $E$. The $\mu$-measurable sets are defined using the Carathéodory criterion.

In case $N$ is a metric space, we say that the image of $f \colon M \to N$ is \emph{essentially separable} if there exists a $\mu$-negligible set $E \subset M$ so that $f( M \setminus E )$ is separable. We say that $f \colon M \to N$ is \emph{$\mu$-measurable} if the image of $f$ is essentially separable and $f^{-1}( U )$ is $\mu$-measurable for every open $U \subset N$.

A mapping is \emph{Borel} if its image is essentially separable and the preimage of every open set is Borel. We say that $\mu$ is a \emph{Borel measure} if every Borel set is $\mu$-measurable and for every $\mu$-measurable set $E \subset M$, there is Borel $B \supset E$ with $\mu( B ) = \mu( E )$. The measures we consider are Borel measures.

We say that $\mu$ is \emph{locally finite} if $M$ can be covered by open sets at which $\mu$ is finite. We say that $\mu$ is \emph{finite} if $\mu( M ) < \infty$.

Given two outer measures, $\mu_1$ and $\mu_2$, we say that $\mu_1 \leq \mu_2$ \emph{in the sense of measures} if $\mu_1( E ) \leq \mu_2( E )$ for every $E \subset M$.

We often use the $n$-dimensional Hausdorff measure $\mathcal{H}^n$ for $n \in \left\{0,1,2,\dots\right\}$. The measure is normalized in such a way that $\mathcal{H}^n$-measure of the $n$-dimensional Euclidean unit ball is equal to the corresponding $n$-dimensional Lebesgue measure $\omega_n$.

Given a Borel map $f \colon M \to N$ and a finite Borel measure $\mu$, the \emph{pushforward} $f_\sharp \mu$ is the Borel measure defined by $f_\sharp\mu( E ) \coloneqq \mu( f^{-1}(E) )$ for $E \subset N$. 

\subsection{Lebesgue spaces}

Given a Banach space $\mathbb{W}$, a $\mu$-measurable $f \colon M \rightarrow \mathbb{W}$ is an element of the space $\mathcal{L}^{\infty}( M, \mathbb{W}; \mu )$ if there is $C > 0$ such that $|f(x)| \leq C$ outside a $\mu$-negligible set. The infimum over such $C$ is denoted by $\|f\|_{ \mathcal{L}^{\infty}( M; \mu ) }$.

Given a locally finite $\mu$ on $M$ and $1 \leq p < \infty$, we say that a $\mu$-measurable $f \colon M \rightarrow \mathbb{W}$ is an element of $\mathcal{L}^{p}( M, \mathbb{W}; \mu )$ if
\begin{equation*}
    \| f \|_{ \mathcal{L}^p( M, \mathbb{W}; \mu ) }
    \coloneqq
    \left(
        \int_{M} |f|^p \, d\mu
    \right)^{1/p}
    <
    \infty.
\end{equation*}

For all $1 \leq p \leq \infty$, we omit $\mathbb{W}$ from the notation when $\mathbb{W} = \mathbb{R}$, and the notation $\mathcal{L}^{p}_+( M; \mu )$ refers to \emph{everywhere} nonnegative elements of $\mathcal{L}^{p}( M; \mu )$.

We let $L^{p}( M, \mathbb{W}; \mu )$ denote the Banach space obtained under the equivalence relation $f_1 \sim f_2$ defined by $\| f_1 - f_2 \|_{ \mathcal{L}^{p}( M, \mathbb{W}; \mu ) } = 0$. The induced norm on the quotient space is denoted by $\| \cdot \|_{ L^p( M, \mathbb{W}; \mu ) }$.

We write $f \in \mathcal{L}^{p}_{loc}( M, \mathbb{W}; \mu )$ in case $M$ can be covered by open sets $( U_i )_{ i = 1 }^{ \infty }$ for which $\chi_{ U_i } f \in \mathcal{L}^{p}( M, \mathbb{W}; \mu )$. Here $\chi_{E}$ refers to the \emph{characteristic function} of $E$: the function being one on $E$ and zero otherwise. The space $L^{p}_{loc}( M, \mathbb{W}; \mu )$ is obtained from $\mathcal{L}^{p}_{loc}( M, \mathbb{W}; \mu )$ by identifying elements whose difference is zero $\mu$-almost everywhere.

The function $\infty \cdot \chi_E$ on $M$ is defined by
\begin{equation*}
    \infty \cdot \chi_E(x)
    \coloneqq
    \left\{
    \begin{split}
        &\infty, \quad&&\text{when $x \in E$},
        \\
        &0, \quad&&\text{otherwise}.
    \end{split}
    \right.
\end{equation*}

\subsection{Sobolev spaces}
In this section we consider a Riemannian $n$-manifold $M$ endowed with the Riemannian length distance $d$ and a locally finite Borel measure $\mu$. We also consider a Banach space $\mathbb{W}$.

When defining $W^{1,p}_{loc}( M, \mathbb{W}; \mu )$, we follow the approach due to Heinonen and Koskela \cite{Hei:Kos:98}, see also \cite{Sha:00,Hei:Kos:Sha:Ty:01,Wil:12,Hei:Kos:Sha:Ty:15}. To this end, given $f \colon M \to \mathbb{W}$, we say that a Borel function $\rho \colon M \to [0,\infty]$ is an \emph{upper gradient} of $f$ if
\begin{equation}\label{eq:uppergradientinequality}
    | f( \gamma(1) ) - f( \gamma(0) ) |
    \leq
    \int_{ \gamma } \rho \, ds,
\end{equation}
for every rectifiable $\gamma \colon [0,1] \to N$, where the right-hand side is the \emph{path integral} of $\rho$ over the length measure of $\gamma$. Given that \eqref{eq:uppergradientinequality} is parametrization-independent, we typically assume that $\gamma$ has constant speed.

We say that $\rho$ is a \emph{locally $p$-integrable $p$-weak upper gradient} of $f$ if $\rho \colon M \to [0,\infty]$ is Borel, $\rho \in \mathcal{L}^{p}_{loc}( M )$, and there is a nonnegative Borel $\widehat{\rho}_0 \in \mathcal{L}^{p}_{loc}( M; \mu )$ so that \eqref{eq:uppergradientinequality} holds for the triple $( f, \gamma, \rho )$ for every rectifiable $\gamma \colon [0,1] \to M$ satisfying $\int_{ \gamma } \widehat{\rho}_0 \, ds < \infty$. The class of locally $p$-integrable $p$-weak upper gradients defined as above coincides with the one in \cite{Hei:Kos:98,Hei:Kos:Sha:Ty:01,Wil:12,Hei:Kos:Sha:Ty:15}.

Whenever $f \colon M \to \mathbb{W}$ has a locally $p$-integrable $p$-weak upper gradient it has a \emph{$p$-minimal one}: there is a unique element in $L^{p}_{loc}( M; \mu )$, denoted by $|df|$, such that every nonnegative Borel representative $\rho$ of $|df|$ is a $p$-weak upper gradient of $f$ and $\rho \leq \rho'$ $\mu$-almost everywhere for every locally $p$-integrable $p$-weak upper gradient $\rho'$ of $f$. Moreover, if a mapping $f_1 \colon M \to \mathbb{W}$ has a locally $p$-integrable $p$-weak upper gradient $\rho_1$ and $f = f_1$ $\mu$-almost everywhere, then $\rho \leq \rho_1$ $\mu$-almost everywhere. 

With the above observation at hand, we define the local Sobolev space $$W^{1,p}_{loc}( M, \mathbb{W}; \mu ) \subset L^{p}_{loc}( M, \mathbb{W}; \mu )$$ as the subset of equivalence classes which have a representative with a locally $p$-integrable $p$-weak upper gradient. We may also unambigously refer to the \emph{$p$-minimal $p$-weak upper gradient} $|df|$ of $f \in W^{1,p}_{loc}( M, \mathbb{W}; \mu )$ as an element of $L^{p}_{loc}( M; \mu )$.

For a complete metric space $N$ and an isometric embedding $\iota \colon N \to \mathbb{W}$ into an injective Banach space $\mathbb{V}$, we define the \emph{local Sobolev space} $$W^{1,p}_{loc}( M, N; \mu ) \coloneqq W^{1,p}_{loc}( M, \mathbb{W}; \mu ) \cap \left\{ f \colon f(x) \in N \,\mu\text{-almost everywhere} \right\}.$$ The space is independent of the embedding and $\mathbb{W}$, cf. \cite[Section 7.1]{Hei:Kos:Sha:Ty:15}.

We obtain the space $W^{1,p}( M, N; \mu ) \subset W^{1,p}_{loc}( M, N; \mu )$ by requiring that $f \in L^{p}( M, \mathbb{W}; \mu )$ and $|df| \in L^{p}( M; \mu )$. Notice that the norm on the former space depends on the embedding, see \cite{Haj:07,Haj:09} and \cite[Sections 7.1 and 7.6]{Hei:Kos:Sha:Ty:15} for related discussion.

We omit $\mu$ from the notation when $M$ is a Riemannian $n$-manifold endowed with the Hausdorff measure $\mathcal{H}^n$.

We prove a key lemma used during the construction of the pushforward of currents. The interested reader can easily modify the following proof to $p$-PI spaces, i.e. complete metric spaces $M$ endowed with a doubling Borel measure $\mu$ and supporting a weak $(1,p)$-Poincaré inequality in the sense of Heinonen and Koskela, see \cite{Hei:Kos:98} for the definitions. The result also generalizes to bounded uniform domains on $p$-PI spaces due to \cite[Theorem 1.8, Proposition 1.9 and Proposition 1.10]{GB:Iko:Zhu:22}. However, we specialize the lemma to our needs.
\begin{lemma}\label{lemm:energydensity}
Let $M$ be a compact Riemannian $n$-manifold, with or without boundary, $\mathbb{V}$ be an injective Banach space, and $f \in W^{1,p}( M, \mathbb{V} )$ for $1 \leq p <\infty$. Then there exists a sequence $( f_j )_{ j = 1 }^{ \infty }$ from $\LIP_{bs}( M, \mathbb{V} )$ satisfying
\begin{equation*}
    \sum_{ j = 1 }^{ \infty }
    \left(
    \| f - f_j \|_{ L^{p}( M, \mathbb{V} ) }
    +
    \| |df| - \lip( f_{j} ) \|_{ L^{p}( M ) }
    \right)
    <
    \infty.
\end{equation*}
In fact, we may assume the existence of finite-dimensional subspaces $\mathbb{W}_j \subset \mathbb{W}_{j+1}$ such that $f_j$ takes values in $\mathbb{W}_j$.
\end{lemma}
During the proof, we use the following definition and the subsequent lemmas. We say that a Banach space $\mathbb{H}$ has the \emph{metric approximation property} if for every compact set $K \subset \mathbb{H}$ and every $\epsilon > 0$, there exists a finite rank linear map $T \colon \mathbb{H} \rightarrow \mathbb{H}$ with operator norm at most one for which
\begin{equation*}
    \sup_{ z \in K } | z - T(z) | \leq \epsilon.
\end{equation*}
\begin{lemma}\label{lem:metrapproximation}
Every injective Banach space $\mathbb{V}$ has the metric approximation property.
\end{lemma}
\begin{proof}[Proof of \Cref{lem:metrapproximation}]
\Cref{lem:inj:infty} reduces the proof to the special case $L^{\infty}(Z)$ for $Z \neq \emptyset$. The claim in that case follows e.g. from \cite[p. 258, (9)]{Ko:79}.
\end{proof}
We also need the following result.
\begin{proposition}[Proposition 4.4, \cite{GB:Iko:Zhu:22}]\label{prop:approximation}
Let $M$ be a complete and separable metric space endowed with a Borel measure $\mu$ that is finite on bounded sets, $\mathbb{V}$ is a Banach space with the metric approximation property, and $1 \leq p < \infty$.

If $f \in W^{1,p}( M, \mathbb{V}; \mu )$ and $\epsilon > 0$, there is $h \in W^{1,p}( M, \mathbb{V} )$ so that
\begin{equation*}
    \| f - h \|_{ L^{p}( M, \mathbb{V}; \mu ) }
    +
    \| |df| - |dh| \|_{ L^{p}( M; \mu ) }
    <
    \epsilon
\end{equation*}
and there is a finite-dimensional subspace $\mathbb{W} \subset \mathbb{V}$ so that the image of each representative of $h$ satisfies $h( M \setminus E ) \subset \mathbb{W}$ for some $\mu$-negligible set $E \subset M$.
\end{proposition}

\begin{proof}[Proof of \Cref{lemm:energydensity}]
We apply \Cref{lem:metrapproximation} and \Cref{prop:approximation} for $\epsilon_j = 2^{-j-1}$. The conclusion is as follows. There is a sequence of finite-dimensional subspaces $\mathbb{W}_{ j } \subset \mathbb{W}_{ j + 1 } \subset \mathbb{V}$ and $h_j \in W^{1,p}( M, \mathbb{V} )$, taking values in $\mathbb{W}_j$ up to a negligible set, such that
\begin{equation*}
    \| f - h_j \|_{ L^{p}( M, \mathbb{V} ) }
    +
    \| |df| - |dh_j| \|_{ L^{p}( M ) }
    \leq
    2^{-j-1}.
\end{equation*}
It follows, e.g. from \cite[Theorem 1.9]{EB:Sou:21} that $\LIP_{bs}( M )$ is norm-dense in $W^{1,p}( M )$ (although the conclusion in our case is classical). Given that $\mathbb{W}_j$ have finite dimension, the norm-density of $\LIP_{bs}( M, \mathbb{W}_j )$ in $W^{1,p}( M, \mathbb{W}_j )$ follows from the corresponding real-valued case by John's theorem and componentwise approximation; see e.g. \cite[Lemma 4.5]{GB:Iko:Zhu:22}. Now \cite[Theorem 1.8]{EB:Sou:21} implies that for $g \in \LIP_{bs}( M, \mathbb{W}_j )$, the pointwise Lipschitz constant $\lip(g)$ and $p$-minimal $p$-weak upper gradient $|dg|$ coincide. Combining this with the stated norm-density shows the existence of $f_j \in \LIP_{bs}( M,  \mathbb{W}_j ) \subset \LIP_{bs}( M, \mathbb{V} )$ for which
\begin{equation*}
    \| h_j - f_j \|_{ L^{p}( M,  \mathbb{V} ) }
    +
    \| |dh_j| - \lip( f_j ) \|_{ L^{p}( M ) }
    \leq
    2^{-j-1}.
\end{equation*}
The claim follows from triangle inequality and from summing over $j$.
\end{proof}

\subsection{Support and a measure-theoretic axiom}\label{sec:technicalassumption}
Given a subset $E$ of a metric space $N$, we say that a Borel measure $\mu$ on $N$ is \emph{concentrated on $E$} if $\mu( N \setminus E ) = 0$.

The \emph{support} of $\mu$ is the set of all $x \in N$ which satisfy $\mu( B( x, r ) ) > 0$ for every $r > 0$. We say that $\mu$ is \emph{compactly supported} if $\mu$ is \emph{concentrated} on a compact set. In particular, the support is compact and $\mu$ is concentrated on its support.

We recall that for any finite Borel measure $\mu$ on $N$, the support is separable, cf. \cite[2.2.16]{Fed:69}. If $\mu$ is concentrated on its support, we may consider $\mu$ as a finite measure on a complete and \emph{separable} space. Then \cite[Theorem 7.1.7]{Bog:07:II} implies the following standard fact.
\begin{lemma}\label{lemm:Radon}
If $N$ is a complete metric space and $\mu$ a finite Borel measure on $N$ that is concentrated on its support, then
\begin{equation*}
    \mu( B )
    =
    \sup\left\{
        \mu( K )
        \colon
        B \supset K \text{ is compact}
    \right\}
    \quad\text{for every Borel set $B \subset N$.}
\end{equation*}
In other words, $\mu$ is Radon.
\end{lemma}
It is clear that if a finite Borel measure $\mu$ on a complete metric space is Radon, then $\mu$ is concentrated on a countable union of compact sets, i.e. on a $\sigma$-compact set. Observe that every finite Borel measure concentrated on a $\sigma$-compact set is concentrated on its support. Hence the following corollary of \Cref{lemm:Radon} holds.
\begin{lemma}\label{lemm:separable}
If $N$ is a complete metric space and $\mu$ a finite Borel measure on $N$, then the following are equivalent:
\begin{enumerate}
    \item $\mu$ is concentrated on its support;
    \item $\mu$ is concentrated on a $\sigma$-compact set;
    \item $\mu$ is Radon;
    \item $\mu$ is concentrated on a closed and separable subset of $N$.
\end{enumerate}
\end{lemma}
We recall the following related fact.
\begin{lemma}[2.2.16, \cite{Fed:69}]\label{lemm:Ulam}
If $N$ is a complete metric space which contains a dense subset whose cardinality is an Ulam number, then every finite Borel measure on $N$ is concentrated on its support.
\end{lemma}
We refer the reader to \cite[2.1.6]{Fed:69} for the definition of Ulam numbers but recall that e.g. all accessible cardinals are Ulam numbers. In particular, $\ell^{\infty}$ satisfies the assumptions of \Cref{lemm:Ulam}.

As Ambrosio and Kirchheim remark, see \cite[p. 12]{AK:00:current}, when developing their theory of metric currents in \emph{every} complete metric space, they make use of the simplifying axiom, consistent with standard ZFC set theory, that the cardinality of \emph{every set} is an Ulam number. We use the same axiom. They remark that even without the axiom, an equivalent theory is obtained in complete metric spaces whose finite Borel measures are concentrated on $\sigma$-compact sets; recall also \Cref{lemm:separable}. To see the relevance of the axiom for our main theorem, see \Cref{rem:remove:technical}.

\subsection{Ambrosio-Kirchheim currents}\label{sec:current:cochain}
We recall the basic definitions and properties of Ambrosio--Kirchheim currents for any complete metric space $N$.

\subsubsection{Currents}
A \emph{$0$-current} is any linear function $T \colon \LIP_b( N ) \rightarrow \mathbb{R}$ for which there is a finite Borel measure $\mu$ with
\begin{equation}\label{eq:0-currentmass}
    | T( g ) |
    \leq
    \int_N |g| \,d\mu
    \quad\text{for every $g \in \LIP_{b}(N)$.}
\end{equation}

If $k \geq 1$, a \emph{$k$-current} is a function $T \colon \LIP_b( N ) \times \LIP( N, \mathbb{R}^{k} ) \rightarrow \mathbb{R}$ satisfying the following axioms:
\begin{enumerate}
    \item (multilinearity): $(g, h) \mapsto T( g, h )$ is multilinear with respect to $g$ and the components of $h$;
    \item ($\text{weak}^{\star}$-continuity): for every $g \in \LIP_b( N )$ and every $h_j \in \LIP( N, \mathbb{R}^k )$  converging to $h \in \LIP( N, \mathbb{R}^k )$ pointwise and $\sup_j \LIP( h_j ) < \infty$, we have
    \begin{equation*}
        T( g, h ) = \lim_{ j \rightarrow \infty } T( g, h_j );
    \end{equation*}
    \item (locality): $T( g, h ) = 0$ if some component of $h$ is constant in an open neighbourhood of $\left\{ g \neq 0 \right\}$;
    \item (finite mass): there is a finite Borel measure $\mu$ on $N$ for which
    \begin{equation}\label{eq:finitemass:axiom}
        |T( g, h )| \leq \int_N |g| \,d\mu
    \end{equation}
    whenever $g \in \LIP_{b}(N)$ and $h = ( h_1, \dots, h_k ) \in \LIP( N; \mathbb{R}^k )$ has $1$-Lipschitz components.
\end{enumerate}

\subsubsection{Mass}
When $k = 0$ or $k \geq 1$ and $T$ is a $k$-current, the \emph{mass measure} $\|T\|$ is the smallest measure satisfying \eqref{eq:0-currentmass} or \eqref{eq:finitemass:axiom}, respectively. We denote
\begin{equation*}
    T \in \mathbf{M}_k(N)
    \quad\text{and}\quad
    M( T ) \coloneqq \| T \|( N ),
\end{equation*}
where $M( T )$ is the \emph{mass} of $T$. We also define $\mathbf{M}_{-1}( N ) = \left\{0\right\}$ to ease later notations.

The space $\mathbf{M}_k( N )$ becomes a vector space when endowed with addition and scalar product defined pointwise. It becomes a Banach space when endowed with the norm $(T,S) \mapsto M( T-S )$.

\subsubsection{Boundary}
We define the \emph{boundary} $\partial T$ of a $k$-current $T$: When $k = 0$ or $k = 1$, respectively, we set $\partial T \equiv 0$ and
\begin{equation*}
    \partial T( g ) \coloneqq T( 1, g )
    \quad\text{for every $g \in \LIP_b( N )$.}
\end{equation*}
Finally, when $k \geq 2$,
\begin{equation*}
    \partial T( g, h ) \coloneqq T( 1, (g,h) )
    \quad\text{for every $(g,h) \in \LIP_b( N ) \times \LIP( N, \mathbb{R}^{k-1} )$.}
\end{equation*}

\subsubsection{Basic properties}
Any $T \in \mathbf{M}_k(  N ) $ can be uniquely extended to $\mathcal{L}^{1}( N; \|T\| ) \times \LIP( N, \mathbb{R}^{k} )$ satisfying the following properties:
\begin{enumerate}
    \item (continuity): $\lim_{ i \rightarrow \infty } T( g_i, h_i ) = T( g, h )$ whenever $g_i \in \mathcal{L}^{1}( N; \|T\| )$ converge to $g$ in $\mathcal{L}^{1}( N; \| T \| )$ and $h_{i} \rightarrow h$ pointwise with $\sup \LIP( h_i ) < \infty$.
    \item (strong locality): $T( g, h ) = 0$ if $\left\{ g \neq 0 \right\} = \bigcup_{ i = 1 }^{ \infty } B_i$ with Borel $B_i$ such that at least one of the components of $h$ vanishes at each $B_i$.
\end{enumerate}
See \cite[Theorem 3.5]{AK:00:current} for the proof.

Given a Borel set $E \subset N$, the \emph{restriction} $T\llcorner{E}$ of a current $T$ is defined as follows:
\begin{equation*}
    T\llcorner{E}( g,h ) = T( \chi_E g, h )
    \quad\text{for every $(g,h) \in \LIP_b( N ) \times \LIP( N, \mathbb{R}^{k} )$}.
\end{equation*}
Observe that the value $T\llcorner{E}( g, h )$ depends only on the values of $g$ and $h$ in the Borel set $E$, by the strong locality of $T$. That is, if $( g', h' ), (g,h) \in \mathcal{L}^{1}( N; \|T\| ) \times \LIP( N, \mathbb{R}^{k} )$ coincide on $E$, we have $T\llcorner{E}(g,h) = T\llcorner{E}(g',h')$.

\subsubsection{Pushforward}
Given a Lipschitz mapping $f \colon N \rightarrow Y$ and $T \in \mathbf{M}_{k}( N )$, the \emph{pushforward} $f_\sharp T$ is defined as follows:
\begin{equation*}
    f_\sharp T( g, h )
    \coloneqq
    T( g \circ f, h \circ f)
    \quad\text{for every $(g,h) \in \LIP_b( Y ) \times \LIP( Y, \mathbb{R}^{k} )$}.
\end{equation*}
We observe the following: When $f \colon N \to Y$ is Lipschitz for a complete metric space $Y$, we have
\begin{equation}\label{eq:masspushforward}
    \| f_\sharp T \|
    \leq
    f_\sharp( \lip(f)^k \|T\| )
    \quad\text{for every $T \in \mathbf{M}_{k}(N)$ and $k \geq 0$,}
\end{equation}
in the sense of measures.

For $k = 0$, \eqref{eq:masspushforward} is clear, so we assume $k \geq 1$. Inequality \eqref{eq:masspushforward} follows by observing that given $h = ( h_1, \dots, h_k ) \in \LIP( N, \mathbb{R}^k )$, we have
\begin{align*}
    | ( f_\sharp T )( g, h ) |
    &\leq
    \int_N ( |g| \circ f ) \left( \prod_{ i = 1 }^{ k } \lip( h_i \circ f) \right) \,d\| T \|
    \quad\text{and}
    \\
    \lip( h_i \circ f)
    &\leq
    \LIP( h_i )
    \lip( f )
    \quad\text{for each $i = 1,2,\dots k$}.  
\end{align*}
See \cite[Lemma 3.10, Eq. (3.5)]{Raj:Wen:13} for a proof of the first inequality; the proof applies the fact that the mass measure of $T$ is concentrated on its support. Notice the slight difference in notation between the cited Lemma and the first inequality above. The second inequality is elementary, so \eqref{eq:masspushforward} is immediate from the definition of mass.

\subsubsection{Subclasses}
When $k \geq 1$ and $( T, \partial T ) \in \mathbf{M}_k(N) \times \mathbf{M}_{k-1}(N)$, we say that $T$ is a \emph{normal $k$-current} and denote $T \in \mathbf{N}_{k}(N)$. Observe also that $\partial( \partial T ) \equiv 0$ for every $T \in \mathbf{N}_{k}(N)$ by the locality axiom of currents. When $\partial T = 0$, we say that $T$ is a \emph{$k$-cycle}. Also, $\mathbf{N}_{0}( N ) = \mathbf{M}_0( N )$, i.e. every $0$-current is a $0$-cycle and $\mathbf{N}_{-1}( N ) = \left\{0\right\}$.

We say that $T \in \mathbf{M}_{k}(N)$ is \emph{$k$-rectifiable} if $\| T \|$ is absolutely continuous with respect to $\mathcal{H}^{k}$ and there is a sequence of Borel sets $B_i \subset \mathbb{R}^k$ and bi-Lipschitz maps $f_i \colon B_i \rightarrow f( B_i ) \subset N$ for which
\begin{equation*}
    \| T \|( N \setminus \bigcup_{ i = 1 }^{ \infty } f_i( B_i ) )
    =
    0.
\end{equation*}
As an example, if $\theta \in L^{1}( \mathbb{R}^{k}; \mathcal{H}^k )$ and $k \geq 1$, then
\begin{equation*}
    \left[ \theta \right]( g, h )
    \coloneqq
    \int_{ \mathbb{R}^k }
        g(x)
        \mathrm{det}( dh )(x)
        \theta(x)
    \,d\mathcal{H}^{k},
    \quad\text{$(g,h) \in \LIP_b( \mathbb{R}^k ) \times \LIP( \mathbb{R}^k, \mathbb{R}^{k} )$,}
\end{equation*}
defines a $k$-rectifiable current. Here $\mathrm{det}( dh )$ is the usual \emph{determinant} of the differential $dh$.

A \emph{$k$-rectifiable} current is \emph{integer rectifiable} if for every Lipschitz $f \colon N \rightarrow \mathbb{R}^{k}$ and for every open set $A \subset N$, there exists an integer-valued $\theta \in \mathcal{L}^{1}( \mathbb{R}^{k}; \mathcal{H}^k )$ for which $f_\sharp( T\llcorner{A} )= \left[ \theta \right]$. We denote such currents by $\mathcal{I}_k( N )$, including $\mathcal{I}_{-1}( N ) = \left\{0\right\}$. We recall a representation result of integer-rectifiable currents in \Cref{sec:earlier}.

An \emph{integral $k$-current} refers to any $T \in \mathcal{I}_k( N ) \cap \mathbf{N}_{k}( N )$ with $\partial T \in \mathcal{I}_{k-1}( N )$, the collection of which we denote by $\mathbf{I}_k( N )$.

The \emph{support} of $T \in \mathbf{M}_{k}(N)$ is the support of $\|T\|$ as a measure. We say that $T$ is \emph{compactly supported}, and denote $T \in \mathbf{M}_{k,c}( N )$, when $\|T\|$ is compactly supported, i.e. concentrated on a compact set. We also denote $\mathbf{N}_{k,c}( N ) \coloneqq \mathbf{M}_{k,c}( N ) \cap \mathbf{N}_{k}(N)$ and $\mathbf{I}_{k,c}( N ) \coloneqq \mathbf{M}_{k,c}( N ) \cap \mathbf{I}_{k}(N)$.

\subsubsection{Extension to locally complete spaces}
We say that a metric space $N$ is \emph{locally complete} if every point has an open neighbourhood whose closure is complete. For such an $N$ and $k \geq 1$, we define $\mathbf{M}_{k,c}( N )$ as the collection of all $T \colon \LIP_{b}( N ) \times \LIP( N; \mathbb{R}^k ) \to \mathbb{R}$ satisfying axioms (1) to (4) as above, with the further assumption that some finite Borel measure $\mu$ satisfying \eqref{eq:finitemass:axiom} is concentrated on a compact set. When $k = 0$, the original definition of a $k$-current is modified by requiring that some measure satisfying \eqref{eq:0-currentmass} is concentrated on a compact set. The results and notation above extend to the subclasses of $\mathbf{M}_{k,c}( N )$.

\subsubsection{Exceptional collections}\label{sec:exceptional}
We define a notion of exceptional collections of normal currents. The following notion is in parts motivated by the modulus of surface families considered by Fuglede \cite{Fug:57}; see also \cite{Raj:Wen:13,Kan:Pry:22}.

Consider a locally complete metric space $M$ endowed with a locally finite Borel measure $\mu$. Given $p \in [1, \infty)$ and a Borel $\rho_0 \in \mathcal{L}^{p}_+( M; \mu )$, we say that
\begin{align*}
    \mathcal{F}( \rho_0 )
    \coloneqq
    \bigcup_{ k \in \left\{ 1, \dots, \lfloor p \rfloor \right\} }
    &\left\{ 
        T\in \mathbf{N}_{k,c}( M )
        \colon
        \int_{N} \rho_0^{ k } \,d\|T\|
        +
        \int_{N} \rho_{0}^{ k-1 } \,d\| \partial T \|
        =
        \infty
    \right\}
    \\
    \cup
    &\left\{ T \in \mathbf{N}_{0,c}( M ) \colon \int_{N} \rho_0^{0} \,d\|T\| = \infty \right\}
\end{align*}
is $p$-exceptional. A countable union of $p$-exceptional collections is contained in a $p$-exceptional collection and since every element of $\mathcal{F}$ is compactly supported, an equivalent definition of $p$-exceptionality is obtained if we only require that $\rho_0 \in \mathcal{L}^{p}_{loc}( M; \mu )$ and $\rho_0$ is a nonnegative Borel function. We consider the direct sum
\begin{equation*}
    \mathbf{N}_{\star,c}( \rho_0 )
    \coloneqq
    \bigoplus_{ k = 0 }^{ \lfloor p \rfloor }
    \mathbf{N}_{k,c}( M ) \setminus \mathcal{F}( \rho_0 )
\end{equation*}
of vector spaces over $\mathbb{R}$, endowed with the usual boundary operation.

We say that $\mathbf{N}_{\star, c}( \rho_0 )$ is a \emph{common refinement} of $\mathbf{N}_{ \star, c }( \rho_1 )$ and $\mathbf{N}_{ \star, c }( \rho_2 )$, for Borel $\rho_1, \rho_2 \in \mathcal{L}^{p}_+( M )$, if $\mathbf{N}_{\star, c}( \rho_0 ) \subset \mathbf{N}_{ \star, c }( \rho_1 ) \cap \mathbf{N}_{ \star, c }( \rho_2 )$.

When we consider integral currents, we simply replace $\mathbf{N}$ by $\mathbf{I}$ in the definitions of this subsection. Notice, however, that $\mathbf{I}_{\star,c}( \rho_0 )$ is a $\mathbb{Z}$-module instead of a vector space over $\mathbb{R}$.

\section{Currents and Sobolev maps}\label{sec:integral}
For the duration of this section, we consider a compact Riemannian $n$-manifold $M$, with or without boundary, $1 \leq p < \infty$ and an injective Banach space $\mathbb{V}$. We endow $M$ with the Hausdorff measure $\mathcal{H}^n$. We recall that $M$ is a $p$-PI space, see \cite[Section 6]{Hei:Kos:98}.

The goal of this section is to simplify the proof of \Cref{thm:pushforward} in \Cref{sec:mainresults:proof}. In particular, we establish some terminology and notation in this section and establish the existence of a \emph{precise} representative of any given Sobolev equivalence class taking values in an injective Banach space.

We recall the following consequence of the $p$-PI property.
\begin{lemma}\label{lemm:decomposition}
Let $M$ be a compact Riemannian manifold, with or without boundary. Then, for every injective Banach space $\mathbb{W}$ and every $H \in W^{1,p}( M, \mathbb{W} )$ for $1 \leq p < \infty$, there is a decomposition $( B_i, H_i )_{ i = 0 }^{ \infty }$ defined as follows: the collection~$\left\{ B_i \right\}_{ i = 0 }^{ \infty }$ is a Borel decomposition of $M$, $B_0$ has negligible measure, and, for every $i \geq 1$, $H_i$ is Lipschitz and an extension of $H|_{ B_i }$, with $B_i$ having positive measure.
\end{lemma}
We refer to such a decomposition $( B_i, H_i )_{ i = 0 }^{ \infty }$ as a \emph{Lusin--Lipschitz decomposition} of $H$. There are no constraints on $H_0$. The proof is based on the existence of Haj\l{}asz gradients \cite{Haj:96} for Banach-valued Sobolev mappings.
\begin{proof}
It suffices to observe that every $H \in W^{1,p}( M, \mathbb{W} )$ has a \emph{Haj\l{}asz} gradient in weak-$\mathcal{L}^p(M)$; see, for example, \cite[Theorem 8.1.49]{Hei:Kos:Sha:Ty:15}. That is, there exist a nonnegative Borel $g$ in weak-$\mathcal{L}^{p}(M)$ and a negligible Borel set $N_0 \subset M$ such that
\begin{equation}\label{eq:Hajlaszgradient}
    | H(x) - H(y) | \leq d(x,y) ( g(x) + g(y) )
    \quad\text{for every $x, y \in M \setminus N_0$.}
\end{equation}
We find a Borel decomposition of $M \setminus N_0$ into a negligible set $C_0$ and a countable number of Borel sets $C_i \subset \left\{ g \leq \lambda_i \right\}$ for some $\lambda_i \in ( 0,\infty )$, with positive measure, for $i \in \left\{1,2,\dots\right\}$. We define the set $B_0 = C_0 \cup N_0$ and $B_i = C_i$ for $i \geq 1$.

Now $H|_{ B_i }$ is Lipschitz for every $i \geq 1$ by \eqref{eq:Hajlaszgradient}, so the claim follows by extending $H|_{ B_i }$ for $i \geq 1$ using \Cref{lemm:mcshane} for $L_i = \LIP( H|_{ B_i } )$.
\end{proof}

\begin{remark}\label{rem:smooth:to:metric}
The construction of the following subsection, \Cref{sec:preciserep}, can be carried out e.g. on bounded Lipschitz domains on Euclidean spaces or even bounded uniform domains in $p$-PI spaces. The key point is that the $M$ in question satisfies the conclusions of Lemmas \ref{lemm:energydensity} and \ref{lemm:decomposition}. These conclusions and the reference measure of $M$ being Radon are the only assumptions we need in \Cref{sec:preciserep}.

\Cref{sec:pushforward:prep} only uses the construction from \Cref{sec:preciserep} and $M$ being locally complete. We prefer to state the proof for compact Riemannian $M$ so that we may focus on the key ideas.

Nevertheless, the aforementioned setting is of interest to us in light of a recent exhaustion result of bounded domains in $p$-PI spaces by bounded uniform domains, see \cite{Raj:21}, and the exhaustion method used in the proof of \Cref{thm:pushforward}.
\end{remark}

\subsection{Precise representatives}\label{sec:preciserep}
We apply \Cref{lemm:energydensity} for $f \in W^{1,p}( M, N )$ for a closed set $N \subset \mathbb{V}$. We are given a sequence $( f_j )_{ j = 1 }^{ \infty }$ from $\LIP_{bs}( M, \mathbb{V} )$ satisfying
\begin{equation}\label{eq:fastconvergence}
    \sum_{ j = 1 }^{ \infty }
    \left(
    \| f - f_j \|_{ L^{p}( M, \mathbb{V} ) }
    +
    \| |df| - \lip( f_{j} ) \|_{ L^{p}( M ) }
    \right)
    <
    \infty.
\end{equation}
This subsection consists of three steps and two properties of interest which we apply in the subsequent subsection.

{\color{blue}Step (1)}: Constructing a representative $\widehat{f}$ of $f$ using $( f_j )_{ j = 1 }^{ \infty }$.

To this end, we consider the Borel function
\begin{equation*}\label{eq:fastconvergence:target}
    \rho_2
    \coloneqq
    |f_1|
    +
    \lip( f_1 )
    +
    \sum_{ j = 1 }^{ \infty } 
    | f_{j+1} - f_{j} |
    +
    | \lip( f_{j+1} ) - \lip( f_j
    ) |
    \in
    \mathcal{L}^{p}_+( M ).
\end{equation*}
Notice that the claimed integrability of $\rho_2$ follows from \eqref{eq:fastconvergence}. When $\rho_2(x) < \infty$, the sequences $( f_j(x) )_{ j = 1 }^{ \infty }$ and $( \lip( f_j )(x) )_{ j = 1 }^{ \infty }$ are Cauchy.

Let $E$ denote the collection of all $x \in M$ for which $( f_j(x) )_{ j = 1 }^{ \infty }$ is not a Cauchy sequence. We denote $\widehat{f}(x) = \lim_{ j \rightarrow \infty }f_j(x)$ in $x \in M \setminus E$ and $\widehat{f}(x) = x_0$ in $E$ for arbitrary $x_0 \in f_1( M )$. The function $\widehat{f}$ is a Borel representative of $f$ since $\| f_j - f \|_{ L^{p}( M, \mathbb{V} ) }$ converges to zero as $j \rightarrow \infty$. The following is immediate.

{\color{blue}Property (1)}: The images of $f_j$, $j \geq 1$, and $\widehat{f}$ lie in some closed and separable subspace $\mathbb{V}_0 \subset \mathbb{V}$.

We use {\color{blue}Property (1)} in the following subsection when applying Prokhorov's theorem.

{\color{blue}Step (2)}: We define $\widehat{\rho}(x) = \lim_{ j \rightarrow \infty } \lip( f_j )(x)$ whenever the limit exists and is finite and $\widehat{\rho}(x) = \infty$ otherwise. Then $\widehat{\rho} \leq \rho_2$ everywhere and $\widehat{\rho}$ is a Borel representative of $|df|$ and a $p$-integrable $p$-weak upper gradient of $\widehat{f}$. 

{\color{blue}Step (2)} is clear by Fuglede's lemma, cf. \cite[Fuglede's Lemma]{Hei:Kos:Sha:Ty:15}, as $\widehat{\rho}$ is the pointwise almost everywhere limit of the upper gradients $\lip( f_j )$ of $f_j$. To this end, we start with the Borel function
\begin{equation}\label{eq:fastconvergence:target:good}
    \rho_1 = \infty \cdot \chi_E + \rho_2.
\end{equation}
Observe that $\rho_1 \geq \widehat{\rho}$ everywhere. If $\gamma \colon [0,1] \to M$ is nonconstant, has constant metric speed and satisfies
\begin{equation}\label{eq:fastconvergence:target:good:curvewise}
    \int_\gamma \rho_1 \,ds < \infty,
\end{equation}
the absolutely continuous paths $f_i \circ \gamma$ form an equicontinuous sequence of functions converging pointwise to $f \circ \gamma$ outside the negligible set $\gamma^{-1}( E )$. Indeed, the equicontinuity follows from the fact that $\rho_1$ is an upper gradient of every $f_i$ for $i \geq 1$, while \eqref{eq:fastconvergence:target:good:curvewise} implies that $\mathcal{H}^{1}( \gamma^{-1}(E) ) = 0$. The equicontinuity of $f_i \circ \gamma$ and the density of $[0,1] \setminus \gamma^{-1}(E)$ in $[0,1]$ imply that $( f_i \circ \gamma(s) )_{ i = 1 }^{ \infty }$ is a Cauchy sequence for every $s \in [0,1]$. Thus $\gamma^{-1}(E) = \emptyset$. Having verified this, the fact that $\lip( f_m ) + \sum_{ j = m }^{ \infty } |\lip( f_{m+1} ) - \lip( f_m ) |$ is an upper gradient of every $f_j$ for $j \geq m$ and is bounded from above by $\rho_1$ implies
\begin{align*}
    | \widehat{f}( \gamma(1) ) - \widehat{f}( \gamma(0) ) |
    &\leq
    \lim_{ m \rightarrow \infty }
    \left(
    \int_\gamma \lip( f_m ) + \sum_{ j = m }^{ \infty } | \lip( f_{j+1} ) - \lip( f_j ) | \,ds
    \right)
    =
    \int_{ \gamma } \widehat{\rho} \,ds.
\end{align*}
Consequently, $\widehat{\rho}$ is a $p$-weak upper gradient of $\widehat{f}$. The remaining properties of $\widehat{\rho}$ are immediate from the defining property of $( f_j, \lip( f_j ) )_{ j = 1 }^{ \infty }$.

{\color{blue}Step (3)}: Constructing a Borel function $\rho_0 \in \mathcal{L}^{p}_+( M )$ defining the $p$-exceptional collection $\mathcal{F}( \rho_0 )$.

We show in the following subsection that the pushforward is well-defined in $\mathbf{N}_{\star,c}( \rho_0 )$.

The construction of $\rho_0$ is done in two pieces using $\rho_1$ from \eqref{eq:fastconvergence:target:good} and an auxiliary Banach space $\mathbb{W}$. We denote $f_0 \coloneqq \widehat{f}$ for convenience.

Let $\mathbb{W}$ denote the collection of all sequences $V \coloneqq ( v_i )_{ i = 0 }^{ \infty }$ such that $v_i \in \mathbb{V}$ and $\| V \| \coloneqq \sup_{ i \geq 0 } | v_i | < \infty$. When we endow $\mathbb{W}$ with componentwise multiplication and summation, $\mathbb{W}$ becomes a Banach space. Since $\mathbb{V}$ is injective, so is $\mathbb{W}$.

We consider the mapping $H = ( f_i )_{ i = 0 }^{ \infty } \colon M \to \mathbb{W}$ defined by
\begin{equation*}
    H(x)
    =
    \left\{
    \begin{split}    
        &( f_0(x) )_{ i = 0 }^{ \infty }
        \quad&&\text{when $x \in E$ and}
        \\
        &( f_0(x) )_{ i = 0 }^{ \infty }
        +
        \lim_{ m \rightarrow \infty }
        \lim_{ j \rightarrow \infty }
        ( f_i(x) - f_j(x) )_{ i = 0 }^{m}
        \quad&&\text{when $x \in M \setminus E$.}
    \end{split}
    \right.
\end{equation*}
Observe that $H$ is Borel measurable as a pointwise limit of Borel measurable mappings, see e.g. \cite[Pettis measurability theorem and Corollary 3.1.5]{Hei:Kos:Sha:Ty:15}. Furthermore, $H(x) = ( f_i(x) )_{ i = 0 }^{ \infty }$ for every $x \in M \setminus E$.

By arguing as in {\color{blue}Step (2)} for a nonconstant $\gamma$, either \eqref{eq:fastconvergence:target:good:curvewise} fails or $\gamma^{-1}( E ) = \emptyset$ and the triples $( f_i, \gamma, \rho_1 )$ satisfy the upper gradient inequality for every $i \geq 0$. Therefore $\rho_1$ is a $p$-integrable upper gradient of $H$. Since also $\|H\| \leq \rho_1$ everywhere, we have $H \in W^{1,p}( M, \mathbb{W} )$.

{\color{blue}Property (2)}: The mapping $H$ has a Lusin--Lipschitz decomposition $( B_i, H_i )_{ i = 0 }^{ \infty }$ with the following properties:
\begin{equation}\label{eq:badpoints:contained}
    \widehat{f}^{-1}( \mathbb{V} \setminus N )
    \cup
    \left\{ \rho_1 = \infty \right\}
    \subset B_0
    \quad
\end{equation}
and for $i \geq 1$, $B_i$ is compact and the Lipschitz $H_i = ( h_{i,j} )_{ j = 0 }^{ \infty }$ are such that $h_{i,0} = \lim_{ j \rightarrow \infty } h_{i,j}$ uniformly on $M$.

Towards establishing {\color{blue}{Property (2)}}, we fix an initial Lusin--Lipschitz decomposition~$( B_i, H_i )_{ i = 0 }^{ \infty }$ of $H$, obtained from \Cref{lemm:decomposition}. We first modify the decomposition $( B_i )_{ i = 0}^{ \infty }$.

The Borel set on the left-hand side of \eqref{eq:badpoints:contained} is $\mathcal{H}^{n}$-negligible. Then \eqref{eq:badpoints:contained} and the compactness of each $B_i$, for $i \geq 1$, can be required, up to enlarging $B_0$, by recalling that each Borel set is a pairwise disjoint union of a negligible Borel set and a sequence of compact sets. We assume from this point onwards that the $( B_i )_{ i = 0 }^{ \infty }$ have already been modified in the manner explained above and proceed with the 
  modification of $( H_i )_{ i = 0 }^{ \infty }$. Applying \Cref{lemm:mcshane} to each $f_j|_{ B_i }$ for $L_i = \LIP( H_i )$ and $i \geq 1$ shows the following:
\begin{equation*}
    \text{$\mathcal{E}( ( f_j )|_{ B_i } )$
    converge uniformly to
    $\mathcal{E}( ( f_0 )|_{ B_i } )$
    for every $i\geq 1$.}
\end{equation*}
Indeed, the pointwise convergence of the $\LIP( H_i )$-Lipschitz $f_{j}|_{ B_i }$ to $f_0|_{ B_i }$ improves to uniform convergence by the compactness of $B_i$, $i \geq 1$. Then the uniform convergence of the extensions follows from \Cref{lemm:mcshane}. {\color{blue}Property (2)} holds after replacing the original $H_i$ by $\left( \mathcal{E}( ( f_j )|_{ B_i } ) \right)_{ j = 0 }^{ \infty }$ for every $i \geq 1$.

Having fixed $( H_i, B_i )_{ i = 0 }^{ \infty }$ satisfying {\color{blue}Property (2)}, we define
\begin{equation}\label{eq:fastconvergence:target:good:last}
    \rho_0 = \infty \cdot \chi_{ B_0 } + \rho_1.
\end{equation}
{\color{blue}Step (3)} is complete. In the following subsection $\rho_0$ will be defined using \eqref{eq:fastconvergence:target:good:last}. Observe the following key properties:
\begin{enumerate}
    \item $( \widehat{f}( x ), \widehat{\rho}(x) )
    =
    \lim_{ j \rightarrow \infty } ( f_j(x), \lip( f_j )(x) )$
    when $\rho_0(x) < \infty$,
    \item $\widehat{f}^{-1}( \mathbb{V} \setminus N )
    \subset
    B_0
    \subset
    \left\{ \rho_0 = \infty \right\}$
    and
    \item
    $\sup_{ j \in \mathbb{N} }
    \max\left\{ 
        |f_j(x)|,
        \lip( f_j )(x),
        |\widehat{f}(x)|,
        \widehat{\rho}(x)
    \right\}
    \leq
    \rho_0(x)$
    for every $x \in M$.
\end{enumerate}

\begin{remark}\label{rem:Sobolev}
The construction above shows the following fact we find curious: Whenever $k \in \left\{1,2,\dots, \lfloor p \rfloor \right\}$ and $T \in \mathbf{N}_{k,c}( \rho_0 )$, we have that
\begin{equation*}
    ( \widehat{f}, \widehat{\rho} )
    =
    \lim_{ j \rightarrow \infty }
    ( f_j, \lip(f_j) )
    \quad\text{in}\quad
    \mathcal{L}^{k}( M, \mathbb{V}; \|T\| ) \times \mathcal{L}^{k}( M; \|T\| ).
\end{equation*}
In particular, when $k \geq 1$, Fuglede's Lemma implies that $\widehat{\rho}$ is a $k$-integrable $k$-weak upper gradient of $\widehat{f}$ when we endow $M$ with the measure $\|T\|$. As a consequence, we have $\widehat{f} \in W^{1,k}( M, N; \| T \| )$.
\end{remark}

\subsection{Construction of the pushforward}\label{sec:pushforward:prep}
We use the assumptions and notation from \Cref{sec:preciserep}.

{\color{blue}Definition:} When $k \in \left\{1, \dots, \lfloor p \rfloor \right\}$ and $T \in \mathbf{N}_{k,c}( \rho_0 )$, the pushforward $\widehat{f}_\star T$ is defined by the formula
\begin{equation*}
    \sum_{ i = 1 }^{ \infty }
    T( \chi_{ B_i } g \circ \widehat{f}, h \circ h_{i,0} ),
    \quad\text{$(g,h) \in \LIP_b(\mathbb{V}) \times \LIP( \mathbb{V}, \mathbb{R}^{k} )$,}
\end{equation*}
and not defined in $\mathbf{N}_{k}( M ) \cap \mathcal{F}( \rho_0 )$. For $k = 0$, we denote
\begin{equation*}
    \sum_{ i = 1 }^{ \infty }
    T( \chi_{ B_i } g \circ \widehat{f} ),
    \quad\text{$g \in \LIP_b(\mathbb{V})$,}
\end{equation*}
in $\mathbf{N}_{0,c}( \rho_0 )$ and do not define the pushforward otherwise. Recall the notations $\mathbb{V}_0$ and $H_i = ( h_{i,j} )_{ j = 0 }^{ \infty }$ from {\color{blue}Properties (1)} and {\color{blue}(2)}, respectively.

{\color{blue}Property (3)}: Whenever $T \in \mathbf{N}_{\star,c}( \rho_0 )$, we have $\widehat{f}_\star T \in \mathbf{N}_{\star}( \mathbb{V} )$ and the pushforward is a pointwise limit of $( f_j )_\sharp T$. We have $\widehat{f}_\star T \in \mathbf{I}_{\star}( \mathbb{V} )$ if $T \in \mathbf{I}_{\star,c}( \rho_0 )$. 

We proceed with the verification of {\color{blue}Property (3)} when $k \geq 1$, the argument requiring some care.

First, by 
\begin{equation*}
    \int_M \rho_0^{k} \,d\|T\| < \infty,
\end{equation*}
the measure $\| T \|$ is concentrated on $M \setminus B_0$, so the pushforward of the measure $\|T\|$ under $\widehat{f}$ or $f_j$, respectively, is concentrated on $\mathbb{V}_0$. Hence the pushforward measures are Radon, recall \Cref{lemm:separable}. The same argument holds when $( k, T )$ is replaced by $( k-1, \partial T )$.

We denote $\mu \coloneqq \widehat{f}_\sharp( \rho_0^k \|T\| )$ and $\nu \coloneqq \widehat{f}_\sharp( \rho_0^{k-1} \|\partial T\| )$. By dominated convergence in $\mathcal{L}^{1}( M; \|T\| )$ and {\color{blue}Step (1)}, $\mu$ is a weak limit of the finite Radon measures $( f_j )_\sharp( \rho_0^k \|T\| )$ on $\mathbb{V}_0$, the respective convergence holding for the boundary mass measures and $\nu$. Thus we may apply Prokhorov's theorem (see e.g. \cite[Theorem 8.6.4]{Bog:07:II}): there exists a sequence of nested compact sets $( K_r )_{ r = 1 }^{ \infty }$ in $\mathbb{V}_0$ for which
\begin{align}\label{eq:nonboundary:tightness}
    \sup_{ j }
       \left( (f_j)_{\sharp} \left( \rho^{k}_0 \| T \| \right) \right)( \mathbb{V} \setminus K_r )
        +
        \mu( \mathbb{V} \setminus K_r )
    &\leq
    2^{-r}
    \quad\text{and}
    \\\label{eq:boundary:tightness}
    \sup_{ j }
       \left( (f_j)_{\sharp} \left( \rho^{k-1}_0 \| \partial T \| \right) \right)( \mathbb{V} \setminus K_r )
        +
        \nu( \mathbb{V} \setminus K_r )
    &\leq
    2^{-r}.    
\end{align}
By applying \eqref{eq:nonboundary:tightness}, we observe from \eqref{eq:masspushforward} that
\begin{align}\label{eq:sizeofthepushforward}
    \| ( f_j )_\sharp T \|( \mathbb{V} \setminus K_r )
    &\leq
    \int_{ f_{j}^{-1}( \mathbb{V} \setminus K_r ) } \lip( f_j )^k \,d \|T\|
    \leq
    \int_{ f_{j}^{-1}( \mathbb{V} \setminus K_r ) } \rho^k_0 \,d \|T\|
    \\ \notag
    &=
    \left( (f_j)_{\sharp} \left( \rho^{k}_0 \| T \| \right) \right)( \mathbb{V} \setminus K_r )
    \leq
    2^{-r}.    
\end{align}
Similarly, \eqref{eq:boundary:tightness} yields
\begin{equation}\label{eq:sizeofthepushforward:boundary}
    \| ( f_j )_\sharp \partial T \|( \mathbb{V} \setminus K_r )
    \leq
    \left( ( f_j )_\sharp \left( \rho_0^{k-1} \|\partial T\| \right) \right)( \mathbb{V} \setminus K_r )
    \leq
    2^{-r}.
\end{equation}
Inequalities \eqref{eq:sizeofthepushforward} and \eqref{eq:sizeofthepushforward:boundary} imply that the sequence of currents $( ( f_j )_\sharp T )_{ j = 1 }^{ \infty }$ is compact, by \cite[Theorem 5.2]{AK:00:current}, and as such, its accumulation points are normal~$k$-currents supported on $\mathbb{V}_0$. In fact, we show that the sequence has a unique accumulation point given by $f_\star T$.

For each $j$ and Borel set $F \subset M$, we have
\begin{equation}\label{eq:ell1-boundedness}
    \left|
        T( \chi_{ F } ( g \circ f_j ), ( h \circ f_j ) )
    \right|
    \leq
    \| g \|_\infty \LIP^k(h) \int_{ F } \rho^{k}_0 \,d\| T \|
    \quad\text{by \eqref{eq:masspushforward}.}
\end{equation}
Inequality \eqref{eq:ell1-boundedness} and dominated convergence in $\mathcal{L}^{1}( M; \| T \| )$ and $\ell_1$, respectively, show
\begin{align}\label{eq:ell1-boundedness:sumwise}
    T( g \circ f_j, h \circ f_j )
    =
    \sum_{ i = 1 }^{ \infty }
    T( \chi_{ B_i } ( g \circ f_j ), ( h \circ f_j ) ).
\end{align}
By the strong locality of $T$, we may replace $h \circ f_j$ by $h \circ h_{i,j}$ in the series above, where $h_{i,j}$ is the $j$'th component of $H_i$ from {\color{blue}Property (2)}. By {\color{blue}Steps (1)} to {\color{blue}(3)} and, in particular, {\color{blue}Property (2)}, $h \circ h_{i,j}$ are uniformly Lipschitz and converge uniformly to $h \circ h_{i,0}$ for every $i \geq 1$. Furthermore, $\chi_{ B_i }( g \circ f_j )$ converge to $\chi_{ B_i }( g \circ \widehat{f} )$ in $\mathcal{L}^{1}( M; \|T\| )$ by dominated convergence. Consequently, by the continuity of currents,
\begin{align*}
    T( \chi_{ B_i } ( g \circ f_j ), ( h \circ f_j ) )
    =
    T( \chi_{ B_i } ( g \circ f_j ), ( h \circ h_{i,j} ) )
    \rightarrow
    T( \chi_{ B_i } ( g \circ \widehat{f} ), ( h \circ h_{i,0} ) ).
\end{align*}
This pointwise convergence, together with \eqref{eq:ell1-boundedness} and \eqref{eq:ell1-boundedness:sumwise}, and dominated convergence in $\ell_1$ imply
\begin{align*}
    \lim_{ j \rightarrow \infty } T( g \circ f_j, h \circ f_j )
    =
    \sum_{ i = 1 }^{ \infty }
    T( \chi_{ B_i } g \circ \widehat{f}, h \circ h_{i,0} ),
    \quad\text{$(g,h) \in \LIP_b(\mathbb{V}) \times \LIP( \mathbb{V}, \mathbb{R}^{k} )$.}
\end{align*}

Observe that if also $T \in \mathbf{I}_{k}( M )$, then $( f_n )_{\sharp}T \in \mathbf{I}_{k}( \mathbb{V} )$ (readily verified using \cite[Theorem 8.8 (ii)]{AK:00:current}). On the other hand, integral currents are closed under pointwise convergence, under the additional assumption of equiboundedness of mass of $( f_n )_{\sharp}T$ and $\partial ( f_n )_{\sharp}T$, cf. \cite[Theorem 8.5]{AK:00:current}. Thus $\widehat{f}_\star T \in \mathbf{I}_{k}( \mathbb{V} )$ whenever $T \in \mathbf{I}_{k,c}( \rho_0 )$ holds.  The remaining case $k = 0$ is similar as above but simpler.

{\color{blue}Property (4)} (mass-energy inequality): For every Borel set $F \subset \mathbb{V}$, we have
\begin{equation}\label{eq:energymass:inequality}
    \| \widehat{f}_\star T \|( F )
    \leq
    \widehat{f}_\star\left( \widehat{\rho}^k \|T\| \right)( F )
    \quad\text{and}\quad
    \| \widehat{f}_\star \partial T \|( F )
    \leq
    \widehat{f}_\star\left( \widehat{\rho}^{k-1} \|\partial T\| \right)( F ).
\end{equation}
Furthermore, the mass measure $\widehat{f}_\star T$ (resp. $\widehat{f}_\star \partial T$) is concentrated on the closed subset $N \cap \mathbb{V}_0$ of the closed and separable subspace $\mathbb{V}_0$ of $\mathbb{V}$.
 
We only prove the claim for $T$ since the claim for the boundary follows by minor modifications. The case $k = 0$ is also simpler so we assume $k \geq 1$.

Whenever $g \in \LIP_b( N )$, dominated convergence shows
\begin{equation*}
    ( |g| \circ f_j ) \lip( f_j )^k
    \rightarrow
    ( |g| \circ \widehat{f} ) \widehat{\rho}^k
    \quad\text{in $\mathcal{L}^{1}( M; \| T \| )$};
\end{equation*}
here we apply {\color{blue}Steps (1)} and {\color{blue}(2)}. Observe that for any $g \in \LIP_{b}( \mathbb{V} )$ and every $h \in \LIP( \mathbb{V}, \mathbb{R}^k )$ whose components are $1$-Lipschitz, {\color{blue}Property (3)} implies
\begin{align*}
    | \widehat{f}_\star T( g, h ) |
    \leq
    \lim_{ j \rightarrow \infty }
    \int_M | g \circ f_j | \lip( f_j )^k \,d\| T \|
    =
    \int_M ( |g| \circ \widehat{f} ) \widehat{\rho}^k \,d\| T \|.
\end{align*}
The definition of mass therefore shows
\begin{equation*}
    \| \widehat{f}_\star T \|
    \leq
    \widehat{f}_{ \star}( \widehat{\rho}^k \| T \| )
    \quad\text{in the sense of measures,}
\end{equation*}
as claimed. The observation that $\| \widehat{f}_\star T \|$ is concentrated on $N \cap \mathbb{V}_0$ is immediate from \eqref{eq:energymass:inequality}.

{\color{blue}Property (5)}: If $( E_i, f_i )_{ i = 0 }^{ \infty }$ is a Lusin--Lipschitz decomposition of $\widehat{f}$, then
\begin{equation*}
    \widehat{f}_\star T(g,h)
    =
    \sum_{ i = 1 }^{ \infty }
    T( \chi_{ E_i } ( g \circ f_i ), ( h \circ f_i ) )
\end{equation*}
outside a $p$-exceptional collection.

We verify the claim outside the collection $\mathcal{F}_0 \coloneqq \mathcal{F}( \max\left\{ \rho_0, \infty \cdot \chi_{ E_0 } \right\} )$. In case $T \in \mathbf{N}_{k,c}( \max\left\{ \rho_0, \infty \cdot \chi_{ E_0 } \right\}  )$, then $\| T \|( E_0 ) + \| \partial T \|( E_0 ) = 0$ and for any $( g, h ) \in \LIP_b( N ) \times \LIP( N, \mathbb{R}^k )$, we have
\begin{align*}
    T( \chi_{B_j}( g \circ \widehat{f} ), ( h \circ h_{0,j} ) )
    &=
    \sum_{ i = 1 }^{ \infty }
    T( \chi_{ E_{i} }\chi_{ B_j } ( g \circ \widehat{f} ), ( h \circ h_{0,j} ) )
    \\
    &=
    \sum_{ i = 1 }^{ \infty }
    T( \chi_{ E_{i} }\chi_{ B_j } ( g \circ \widehat{f} ), h \circ f_i ),
\end{align*}
where the first equality applies dominated convergence in $\mathcal{L}^{1}( M; \|T\| )$ and finite additivity of $T$, together with dominated convergence in $\ell^{1}$, while the second equality follows from the strong locality of $T$. We may sum over $j \geq 1$ and exchange the order of summations over $j$ and $i$ due to Fubini's theorem and absorb the sum over $j$ inside the last term. The steps are justified by \eqref{eq:ell1-boundedness}. {\color{blue}Property (5)} follows.

{\color{blue}Property (6)}: We claim that the pushforward is well-defined for every representative $\widehat{f}_1$ of $f$ with a $p$-minimal $p$-weak upper gradient $\widehat{\rho}_1$ outside a $p$-exceptional collection $\mathcal{F}_1$.

To this end, let $F$ denote any $\mathcal{H}^n$-negligible Borel set containing the points for which $\widehat{f}_1 \neq \widehat{f}$ or $\widehat{\rho}_1 \neq \widehat{\rho}$. We define $\mathcal{F}_1 = \mathcal{F}( \max\left\{ \rho_0, \infty \cdot \chi_{ F } \right\} )$, and deduce for every $T \in \mathbf{N}_{k,c}( \max\left\{ \rho_0, \infty \cdot \chi_{ F } \right\} )$ that $\widehat{f} = \widehat{f}_1$ and $\widehat{\rho} = \widehat{\rho}_1$ $( \|T\| + \|\partial T\| )$-almost everywhere. This guarantees
\begin{equation*}
     \widehat{f}_\sharp( \widehat{\rho}^k \| T \| )
    =
    ( \widehat{f}_1 )_\sharp( \widehat{\rho}^k_1 \| T \| )
    \quad\text{and}\quad
    \widehat{f}_\sharp( \widehat{\rho}^{k-1} \| \partial T \| )
    =
    ( \widehat{f}_1 )_\sharp( \widehat{\rho}^{k-1}_1 \| \partial T \| ).
\end{equation*}
With these equalities at hand, {\color{blue}Property (4)} justifies the definition $( \widehat{f}_1 )_\star \coloneqq \widehat{f}_\star$ for such $T$. In case we already had a definition for $( \widehat{f}_1 )_\star$ satisfying {\color{blue}Properties (4)} and {\color{blue}(5)} above, similar reasoning as in the proof of {\color{blue}Property (5)} shows that the two definitions coincide outside a $p$-exceptional collection. In this sense, the obtained pushforward is independent of the approximating sequence of Lipschitz mappings we started from. In particular, {\color{blue}Property (6)} follows.

\begin{remark}\label{rem:remove:technical}
When deriving the results of this section, the axiom regarding the nonexistence of non-Ulam cardinal numbers from \Cref{sec:technicalassumption} is not necessary. Indeed, \cite[Lemma 3.10, Eq. (3.5)]{Raj:Wen:13} is used in the proof of \eqref{eq:sizeofthepushforward} and \eqref{eq:sizeofthepushforward:boundary}, applied to $T \in \mathbf{N}_{k,c}(M)$ for locally complete and \emph{separable} $M$, so $\|T\|$ is concentrated on its support. In that setting the proof of \cite[Lemma 3.10, Eq. (3.5)]{Raj:Wen:13} goes through without the axiom.

Inequalities \eqref{eq:nonboundary:tightness} and \eqref{eq:boundary:tightness} imply that the mass measures of the normal currents $\widehat{f}_\star T$ and $( f_j )_{\star} T$ of interest and their boundaries are concentrated on the $\sigma$-compact set $\bigcup_{ r = 1 }^{ \infty }K_r$ in the closed and separable subspace $\mathbb{V}_0$ of $\mathbb{V}$; hence we may consider the currents defined only on $\mathbb{V}_0$ and apply \cite[Theorems 8.5 and 8.8 (ii)]{AK:00:current} only for separable Banach spaces. Thus the axiom is not needed in this case.
\end{remark}

\section{Pushforward and pullback}\label{sec:mainresults:proof}
In this section we prove \Cref{thm:pushforward:intro}. In fact, we prove the following stronger statement.

\begin{theorem}\label{thm:pushforward}
Suppose that $M$ is a Riemannian manifold, with or without boundary, $N$ is a closed subset of an injective Banach space $\mathbb{V}$, and $f \in W^{1,p}_{loc}( M, N )$ for $1 \leq p < \infty$.

Then there exist a Borel representative $( \widehat{f}, \widehat{\rho} )$ of $( f, |df| )$ and Borel $\widehat{\rho}_0 \in \mathcal{L}^{p}_+( M )$ together with a chain morphism
\begin{equation}\label{eq:pushforward}
    \widehat{f}_\star \colon ( \mathbf{N}_{\star, c}( \widehat{\rho}_0 ), \partial ) \to ( \mathbf{N}_{\star}( N ), \partial )
\end{equation}
with the following properties: there is a Lusin--Lipschitz decomposition of $(  \widehat{B}_i, \widehat{f}_i )_{ i = 0 }^{ \infty }$ of $\widehat{f}$ with
\begin{equation}\label{eq:lusinlipschitz}
    \widehat{f}^{-1}( \mathbb{V} \setminus N )
    \subset
    B_0
    \subset
    \left\{ \widehat{\rho}_0 = \infty \right\},
    \quad
    \widehat{f}_\star T
    =
    \sum_{ i = 1 }^{ \infty }
    ( \widehat{f}_i )_{\sharp}\left( T\llcorner{B_i} \right),
\end{equation}
and
\begin{equation}\label{eq:mass-energy}
    \| \widehat{f}_\star T \|( E )
    \leq
    \int_{ f^{-1}(E) } \widehat{\rho}^k \,d\|T\|
    <
    \infty
    \quad\text{for Borel $E \subset N$},
\end{equation}
whenever $T \in \mathbf{N}_{k,c}( \widehat{\rho}_0 )$ and $k \in \left\{ 0,1,\dots, \lfloor p \rfloor \right\}$. Furthermore, $\widehat{f}_\star$ restricts to a chain morphism
\begin{equation*}
    \widehat{f}_\star \colon ( \mathbf{I}_{\star, c}( \widehat{\rho}_0 ), \partial ) \to ( \mathbf{I}_{\star}( N ), \partial ).
\end{equation*}

If another tuple $( (\widetilde{f}, \widetilde{\rho}), \widetilde{\rho}_0, \widetilde{f}_\star, ( \widetilde{B}_i, \widetilde{f}_i )_{ i = 0 }^{\infty} )$ satisfies the conclusions \eqref{eq:pushforward}, \eqref{eq:lusinlipschitz} and \eqref{eq:mass-energy} in $\mathbf{N}_{\star,c}( \widetilde{\rho}_0 )$, the chain morphisms $\widehat{f}_\star$ and $\widetilde{f}_\star$ coincide on a common refinement of $( \mathbf{N}_{\star,c}( \widehat{\rho}_0 ), \partial )$ and $( \mathbf{N}_{\star,c}( \widetilde{\rho}_0 ), \partial )$.
\end{theorem}
\Cref{thm:pushforward:intro} is an immediate consequence of \Cref{thm:pushforward} above. The series representation in \eqref{eq:lusinlipschitz} should be understood as follows:
\begin{equation*}
    \widehat{f}_\star T( g, h )
    =
    \lim_{ N \rightarrow \infty }
    \sum_{ i = 1 }^{ N } ( \widehat{f}_i )_\sharp( T\llcorner{B_i} )(g,h),
    \quad\text{$(g,h) \in \LIP_{b}( \mathbb{V} ) \times \LIP( \mathbb{V}, \mathbb{R}^k )$,}
\end{equation*}
and that $\sum_{ i = 1 }^{ \infty } |( \widehat{f}_i )_\sharp( T\llcorner{B_i} )(g,h)| < \infty$. Notice that the summation starts from $1$ instead of $0$. That is, the map $\widehat{f}_0$ plays no role in the claim.

\begin{proof}[Proof of \Cref{thm:pushforward}]
The proof has two parts.

{\color{blue}Case (1)}: $M$ is compact.

We may apply the results from \Cref{sec:integral} somewhat directly. Given the technicality of the statement, let us provide the details.

First, \Cref{lemm:energydensity} shows that $f$ satisfies the assumptions of \Cref{sec:preciserep}. Let $\widehat{f}$ and $\widehat{\rho}$ denote the precise representatives of $f$ and its $p$-minimal $p$-weak upper gradient, respectively, constructed in {\color{blue}Steps (1)} to {\color{blue}(3)}. The pushforward is now defined as in \Cref{sec:pushforward:prep}.

Given $\rho_0 \in \mathcal{L}^{p}_+( M )$ as in \eqref{eq:fastconvergence:target:good:last}, we obtain the $p$-exceptional collection $\mathcal{F}( \rho_0 )$. Observe that $\mathbf{N}_{k,c}( \rho_0 )$ is closed under summation and scaling, and
\begin{equation*}
    \partial( \mathbf{N}_{k,c}( \rho_0 ) )
    \subset
    \mathbf{N}_{k-1,c}( \rho_0 ),
    \quad\text{for every $k \in \left\{0,1,\dots, \lfloor p \rfloor\right\}$}.
\end{equation*}
The boundary property also holds when $\mathbf{N}_{\star,c}$ is replaced by $\mathbf{I}_{\star,c}$. Observe, however, that $\mathbf{I}_{k,c}( \rho_0 )$ is (only) a $\mathbb{Z}$-module.

We recall from {\color{blue}Property (4)} that the pushforward current is concentrated (and therefore supported) on a separable and closed subset of $N$, in particular, in $N$. Any pair of pushforwards induced by representatives of $f$ satisfying \eqref{eq:lusinlipschitz} and \eqref{eq:mass-energy} coincide outside a common refinement, as a consequence of {\color{blue}Properties (5)} and {\color{blue}(6)}. {\color{blue}Case (1)} follows.

{\color{blue} Case (2)}: $M$ is not compact.

Fix a smooth function $r \colon M \rightarrow \mathbb{R}$ such that $r^{-1}( (-\infty, t] )$ is compact for every $t$, see e.g. \cite[Proposition 2.28]{Lee:21}. Fix a strictly increasing sequence $( t_m )_{ m = 1 }^{ \infty }$ of positive numbers such that $t_m$ is not a critical value of $r$ or of $r|_{ \partial M }$ and $t_m \rightarrow \infty$. Then $M_m \coloneqq r^{-1}( (-\infty, t_m] )$ form a sequence of compact Riemannian $n$-submanifolds of $M$ exhausting $M$.

We apply the compact case, namely {\color{blue}Case (1)}, to each of $M_m$ and $f|_{ M_m }$, and obtain $\rho_{0,m} \in \mathcal{L}^{p}( M_m )$ and a Lusin--Lipschitz decomposition $( \widehat{B}_{i,m}, \widehat{f}_{i,m} )_{ i = 0 }^{ \infty }$ of a representative $\widehat{f}_m$ of $f|_{ M_m }$. We also find a representative $\widehat{\rho}_m \in \mathcal{L}^p( M_m )$ of the $p$-minimal $p$-weak upper gradient of $\widehat{f}_m$ satisfying the conclusion \eqref{eq:mass-energy}.

We apply the locality of the $W^{1,p}$-spaces, cf. \cite[Corollary 3.9]{Wil:12}, to observe that
\begin{equation*}
    \bigcup_{ m = 3 }^{ \infty }
    \bigcup_{ j = m-1 }^{ \infty }
    \left\{ x \in M_{ m-2 } \colon ( \widehat{f}_m, \widehat{\rho}_m )(x) \neq ( \widehat{f}_j, \widehat{\rho}_j )(x) \right\}
\end{equation*}
has $\mathcal{H}^{n}$-negligible measure, so the set is contained in a $\mathcal{H}^{n}$-negligible Borel set $E_0$.

Now we define $\widehat{f}|_{ M_{1} } \coloneqq \widehat{f}_{ 3 }|_{ M_1 }$ and $\widehat{f}|_{ M_{m-2} \setminus M_{m-3} } \coloneqq \widehat{f}_{ m }|_{ M_{m-2} \setminus M_{m-3} }$ for every $m \geq 4$. We define $\widehat{\rho}$ similarly using $\widehat{\rho}_m$. The map $\widehat{f}$ has a locally $p$-integrable $p$-weak upper gradient $\widehat{\rho}$ since $$C_0 \coloneqq \bigcup_{ m = 3 }^{ \infty } \bigcup_{ j = m-1 }^{ \infty } \left\{ x \in M_{m-2} \colon \widehat{f}_{m}(x) \neq \widehat{f}_{j}(x) \right\}$$ has negligible Sobolev $p$-capacity, cf. \cite[Proposition 7.1.31, Corollary 7.2.1]{Hei:Kos:Sha:Ty:15}.

We construct a Lusin--Lipschitz decomposition for $\widehat{f}$. To this end, we denote $B_0 = E_0 \cup \bigcup_{ m = 1 }^{ \infty } \widehat{B}_{0,m}$ and fix a bijection $\phi \colon \mathbb{N} \rightarrow \mathbb{N}^2$ and set for each $j \geq 1$,
\begin{equation*}
    ( B_j, a_j )
    \coloneqq
    \left( \widehat{B}_{ \phi(j) } \setminus \bigcup_{ l = 1 }^{ j -1 } B_l, \widehat{f}_{ \phi(j) } ) \right).
\end{equation*}
We apply \Cref{lemm:mcshane} for each $a_j|_{ B_j } = \widehat{f}_{ \phi(j) }|_{ B_j } = \widehat{f}|_{ B_j }$ for the Lipschitz constant $L = \LIP( \widehat{f}|_{ B_j } )$ and obtain a Lipschitz extension $g_j \colon M \rightarrow \mathbb{V}$ of $\widehat{f}|_{ B_j }$ whose supremum norm is the same as $\widehat{f}|_{ B_j }$. In particular, $\| a_j \|_\infty \geq \| g_j \|_{\infty}$. Lastly, we fix an arbitrary Lipschitz function $g_0 \colon M \rightarrow \mathbb{V}$. Now $( B_j, g_j )_{ j = 0 }^{ \infty }$ is a Lusin--Lipschitz decomposition of $\widehat{f}$.

For every $m \in \mathbb{N}$, we identify $\rho_{0,m}$ with its zero extension to $M$ and consider the function $\infty \cdot \chi_{ B_0 } + \sup_{ j \leq m } \rho_{0,j}$ in place of $\rho_{0,m}$. We use the same notation $\rho_{0,m}$ for the new function for every $m \in \mathbb{N}$.

Next, we set
\begin{equation*}
    \widehat{\rho}_0
    \coloneqq
    \sum_{ m = 3 }^{ \infty }
    2^{-m}\frac{ \rho_{0,m} }{ 1 + \|\rho_{0,m}\|_{ \mathcal{L}^p( M ) } }.
\end{equation*}
The function is Borel and satisfies $\widehat{\rho}_0 \in \mathcal{L}^{p}_+( M )$. Moreover, $\mathbf{N}_{\star,c}( \widehat{\rho}_0 )$ is a common refinement of any pair $\mathbf{N}_{\star,c}( \rho_{0,m} )$ and $\mathbf{N}_{\star, c}( \rho_{0,j} )$ for $m, j\geq 4$.

Given any $T \in \mathbf{N}_{\star,c}( \widehat{\rho}_0 )$, there is $m \geq 4$ such that $T$ is supported in $M_{m-2}$. By the definition of $B_0 \supset E_0$ and as $\infty = \rho_{0,m}$ on $B_0$, we have that $\widehat{f}_{m} = \widehat{f}_{j}$ and $\widehat{\rho}_m = \widehat{\rho}_{j}$ outside the $( \|T\| + \| \partial T \| )$-negligible set $B_0$ for every $j \geq m-1$. Then (the proof of) {\color{blue}Property (5)} guarantees that $( \widehat{f}_m )_\star T$ and $( \widehat{f}_{j} )_\star T$ coincide when $j \geq m-1$. Therefore, we may define
\begin{equation*}
    \widehat{f}_\star T
    =
    \lim_{ m \rightarrow \infty }
    ( \widehat{f}_m )_\star T
    \quad\text{for every $T \in \mathbf{N}_{\star,c}( \widehat{\rho}_0 )$.}
\end{equation*}
Since the Lusin--Lipschitz property \eqref{eq:lusinlipschitz} holds for every $\widehat{f}_m$, the strong locality of currents and the construction of the Lusin--Lipschitz decomposition $( B_j, g_j )_{ j = 0 }^{ \infty }$ imply that $\widehat{f}$ satisfies \eqref{eq:lusinlipschitz}. Lastly, \eqref{eq:mass-energy} follows for $\widehat{\rho}$ since the corresponding property holds for each $m \in \mathbb{N}$. The claim follows.
\end{proof}

\begin{proof}[Proof of \Cref{thm:pullback}]
We continue with the notation from the proof of \Cref{thm:pushforward}. More specifically, from {\color{blue}Case (2)}. This defines the $p$-exceptional collection $\mathcal{F}( \widehat{\rho}_0 )$ and the representative $( \widehat{f}, \widehat{\rho} )$ of $( f, |df| )$.

Next, given $\Omega \in \mathcal{L}_{\infty}( \mathbf{I}_k(N) )$ with upper norm $\rho$, we have
\begin{equation*}
    | \Omega( \widehat{f}_\star T ) |
    \leq
    \int_{ N } \rho \,d\| \widehat{f}_\star T \|
    \leq
    \int_{ M } ( \rho \circ \widehat{f} )\widehat{\rho}^k \,d\|T\|,
    \quad T \in \mathbf{I}_{\star,c}( \widehat{\rho}_0 ),
\end{equation*}
where the latter inequality is immediate from the mass-energy inequality \eqref{eq:mass-energy}. The claim follows.
\end{proof}

\section{Representing integer-rectifiable currents}\label{sec:earlier}
 
\subsection{Vectors and covectors}
Let $\mathbb{V}$ be a Banach space. For each $k = 0,1\dots$, we let $\bigwedge_n \mathbb{V} $ denote the collection of all $n$-vectors. We say that $\tau \in \bigwedge_k \mathbb{V} $ is \emph{simple} if $\tau = v_1 \wedge \dots \wedge v_n$ for some $v_1, \dots, v_n \in \mathbb{V}$. When $\tau$ is simple, $\mathrm{Span}( \tau )$ is the linear subspace span by the vectors $v_1, \dots, v_n$. For $n \in \mathbb{N}$, the set $\bigwedge^n \mathrm{Span}( \tau )$ consists of all $n$-covectors on the subspace $\mathrm{Span}( \tau )$.

We say that $\tau = v_1 \wedge \dots \wedge v_n$ has \emph{unit norm} if $v_1, \dots, v_n \in \mathbb{V}$ can be chosen in such a way that the linear map
\begin{equation*}
    L \colon \mathbb{R}^n \to \mathbb{V},
    \quad\text{where }
    L( x )
    =
    \sum_{ i = 1 }^{ n } x_i v_i
\end{equation*}
has unit Jacobian. That is, we have
\begin{equation*}
    \mathcal{H}^n( L( E ) )
    =
    \mathcal{H}^n( E )
    \quad\text{for every Borel set $E \subset \mathbb{R}^n$;}
\end{equation*}
we refer the reader to \cite[Section 9]{AK:00:current} for well-posedness.

Given a Borel set $E \subset \mathbb{V}$, a map $\tau \colon E \rightarrow \bigwedge_{n} \mathbb{V}$ is \emph{measurable} if $\tau(x) = \tau_1( x ) \wedge \dots \wedge \tau_n(x)$ for some Borel maps $\tau_i \colon E \rightarrow \mathbb{V}$. In this case, we say that $\tau$ is a \emph{measurable section} of $\bigwedge_n \mathbb{V}$.

\subsection{Integer-rectifiable currents}
We consider a closed subset $N$ of an injective Banach space $\mathbb{V}$ and $T \in \mathbf{I}_n(N)$. Since $T$ is concentrated on its (separable) support, we might as well assume that $N$ is separable. Hence we may isometrically embed $N$ into $\ell^{\infty}$ leading to the simplification $\mathbb{V} = \ell^{\infty}$. The assumption $\mathbb{V} = \ell^{\infty}$ allows us to apply the results from \cite[Section 9]{AK:00:current}.

We recall from \cite[Theorem 9.1]{AK:00:current} that every $T \in \mathbf{I}_{n}(N)$ can be represented as follows: there is an $n$-rectifiable Borel set $E_T \subset N$, an integer-valued Borel function $\theta_T \colon E_T \to (0, \infty)$, and a measurable section of unit simple $n$-vectors $\tau_T \colon E_T \to \bigwedge_k\mathbb{V}$ for which the subspace $\mathrm{Span}( \tau_T(x) )$ is the \emph{approximate tangent space} of $E_T$ at $\mathcal{H}^{n}$-almost every $x \in E_T$, and
\begin{equation}\label{eq:representation:current}
    T( g, h )
    =
    \int_{ E_T  }
        g(x)
        \langle \wedge_n d^{E_T}h, \tau_T \rangle(x)
        \theta_T(x)
    \,d\mathcal{H}^n(x).
\end{equation}
Here $\wedge_n d^{E_T}h(x) \colon \bigwedge_n( \mathrm{Span}( \tau_T(x) ) ) \to \bigwedge_n( \mathbb{R}^n ) \sim \mathbb{R}$ is the $n$-covector induced by the \emph{tangential differential $d^{E_T} h$} of $h$ on $\mathrm{Span}( \tau_T(x) )$, and $\langle \cdot, \cdot \rangle$ is the standard duality pairing of $n$-vectors and $n$-covectors on $\mathrm{Span}( \tau_T(x) )$. The sets and functions are unique up to a $\mathcal{H}^n$-negligible set \cite[Theorem 9.1]{AK:00:current}.

The mass measure $\|T\|$ admits the representation
\begin{equation}\label{eq:massmeasure}
    \| T \|( E )
    \coloneqq
    \int_{ E_T \cap E}
        \theta_T(x)
        \lambda_{T}( x )
        \,d\mathcal{H}^n(x)
    \quad\text{for Borel $E \subset N$,}
\end{equation}
where $\lambda_T( x )$ is a factor depending only on the span of $\tau_T(x)$. In fact, we have
\begin{equation}
    \lambda_{T}( x )
    =
    \frac{ 2^n }{ \omega_n }
    \sup
    \left\{
        \frac{ \mathcal{H}^{n}( B_1 ) }{ \mathcal{H}^{n}( R ) }
        \colon
        \mathrm{Span}( \tau_T(x) )
        \supset
        R
        \supset
        B_1
    \right\},
\end{equation}
where $B_1$ is the unit ball in $\mathrm{Span}( \tau(x) )$ and $R$ is any linear image of the $n$-dimensional cube $[-1,1]^n$ containing the unit ball $B_1$. See \cite[Theorem 9.5]{AK:00:current} for the proof of these facts.

When the norm in $\mathrm{Span}( \tau_T(x) )$ is induced by an inner product, we have $\lambda_T( x ) = 1$ and, more generally, we have the bounds $n^{-n/2} \leq \lambda_{T}(x) \leq 2^{n}/\omega_n$, see \cite[Lemma 9.2]{AK:00:current}. In particular, when $N$ is a subset of a Riemannian manifold, we have $\lambda_{T} \equiv 1$ for all integer-rectifiable $n$-currents supported on $N$.

\subsection{Cochains from differential forms}
Using \eqref{eq:representation:current} and \eqref{eq:massmeasure} together with the definition of comass from \eqref{eq:comass}, we see that on a Riemannian manifold $N$ the comass of any bounded differential form defines an upper norm for the cochain $\Omega$ induced by the differential form $\omega$. To be precise, let $\omega$ be a bounded and smooth differential $n$-form on a complete Riemannian manifold $N$ for $n \geq 1$, and consider
\begin{equation}\label{eq:actiondefinition}
    \Omega_{\omega}( T )
    \coloneqq
    \int_{ E_T } \langle \omega, \tau_T \rangle(x) \theta_T(x) \,d\mathcal{H}^n(x)
    \quad\text{for $T \in \mathcal{I}_{k}(N)$,}
\end{equation}
with the understanding that $N$ is isometrically embedded into $\ell^{\infty}$. Then
\begin{equation*}
    | \Omega_\omega(T) |
    \leq
    \int_{ E_T } \|\omega\| \,d\|T\|
    \quad\text{for the comass of $\omega$.}
\end{equation*}

\begin{lemma}\label{lemm:local:to:global}
Let $N$ be a complete Riemannian manifold, with or without boundary, and suppose that an $n$-form $\omega$ and its exterior derivative $d\omega$ are smooth and bounded, with $n \geq 1$. Then
\begin{equation*}
    \Omega_{ \omega }( \partial S )
    =
    \Omega_{ d\omega }( S )
    \quad\text{for every $S \in \mathbf{I}_{n+1}(N)$.}
\end{equation*}
In particular, $\delta( \Omega_{\omega}|_{ \mathbf{I}_{n}(N) } ) = \Omega_{ d\omega }|_{ \mathbf{I}_{n+1}(N) }$.
\end{lemma}

We use the following lemma during the proof.
\begin{lemma}\label{lem:exhaustion}
Let $N$ be a complete Riemannian manifold, with or without boundary. Then there is a smooth $1$-Lipschitz $g \colon N \to \mathbb{R}$ for which the set $g^{-1}( (-\infty,r] )$ is compact for every $r \in \mathbb{R}$.
\end{lemma}
\begin{proof}[Proof of \Cref{lem:exhaustion}]
When $N$ has no boundary, by smoothening the function $h(x) = d( x, x_0 )/2$, $x_0 \in N \setminus \partial N$, we conclude the existence of a smooth $1$-Lipschitz function $g \colon N \to \mathbb{R}$ for which the set $g^{-1}( (-\infty, r] )$ is compact for every $r \in \mathbb{R}$. When $N$ does not have boundary, the existence is justified carefully e.g. in \cite[Corollary 4]{Gree:Wu:79}.

When $N$ does have boundary, we may reduce to the case above as follows. There is a smooth $1$-Lipschitz embedding $\pi \colon N \to N'$ that is a local isometry of the seam and so that $N'$ is a complete Riemannian manifold that has no boundary, after which we construct $g'$ as above for $N'$ instead of $N$ and denote $g = g' \circ \pi$.

The existence of $( N', \pi )$ can be deduced by gluing to $N$ a collar neighbourhood $f \colon [0,1) \times \partial N \rightarrow U$ of $\partial N$ along $\partial N$, then smoothly extending the Riemannian tensor over the seam, and then carefully modifying the obtained tensor near the seam. The interested reader may consult \cite{Pigo:Vero:20} for the details.
\end{proof}

\begin{proof}[Proof of \Cref{lemm:local:to:global}]
Let $g$ be as in \Cref{lem:exhaustion}, and fix any $2$-Lipschitz smooth $h \colon \mathbb{R} \to [0,1]$ that satisfies $h( r ) = 1$ if $r \leq 1$ and $h(r) = 0$ if $r \geq 2$. Now $\psi_m(z) \coloneqq h( 2^{-m}g(z) )$, $z \in N$, is a compactly supported smooth $2^{1-m}$-Lipschitz function, the sequence of which converges to the constant function one uniformly on compact sets and whose differential converges to zero uniformly on compact sets.

Next, we define $\omega_m \coloneqq \psi_m \omega$. Dominated convergence in $\mathcal{L}^{1}( N; \| T \| )$ and $\mathcal{L}^{1}( N; \|S\| )$, respectively, imply that
\begin{align*}
    \Omega_{ \omega }( T )
    &=
    \lim_{ m \rightarrow \infty }
    \Omega_{ \omega_m }( T )
    \quad&&\text{for every $T \in \mathbf{I}_{n}(N)$}
    \quad\text{and}\quad
    \\
    \Omega_{ d\omega }( S )
    &=
    \lim_{ m \rightarrow \infty }
    \Omega_{ d \omega_m }( S )
    \quad&&\text{for every $S \in \mathbf{I}_{n+1}(N)$.}
\end{align*}
Thus the claim follows if we are able to show for every $S \in \mathbf{I}_{n+1}(N)$ that
\begin{equation}\label{eq:compactsupport:case}
    \Omega_{ d\tau }( S ) = \Omega_{ \tau }( \partial S )
    \quad\text{for every compactly-supported smooth $n$-form $\tau$}.
\end{equation}
By the finite additivity of the both sides of \eqref{eq:compactsupport:case} with respect to $\tau$, it suffices to establish \eqref{eq:compactsupport:case} for $\tau$ compactly-supported in $U$ for an arbitrary chart $\phi \colon U \to [0,\infty) \times \mathbb{R}^{ \mathrm{dim}(N) - 1 }$, with a bounded differential. This observation and the finite additivity reduce the claim to the special case $\tau = \alpha \phi^{\star}( dx^{I} )$ with $\alpha$ being smooth and compactly supported in $U$ and $dx^{I}$ being the wedge product of $dx^{ i_1 }, \dots, dx^{ i_{n} }$ for an increasing multi-index $I = ( i_1, \dots, i_{n} )$.

Since $d\tau = d\alpha \wedge \phi^{ \star }( dx^{I} )$, we may apply \eqref{eq:actiondefinition} to $d\tau$ and $S$, together with the definition of $\partial S$, to deduce 
\begin{equation*}
    \Omega_{ d\tau }( S )
    =
    S( 1, (\alpha, \phi_{ i_1 }, \dots, \phi_{ i_n }) )
    =
    \partial S( \alpha, ( \phi_{ i_1 }, \dots, \phi_{ i_n } ) ).
\end{equation*}
Next, by applying \eqref{eq:actiondefinition} to $\tau$ and $\partial S$, we deduce
\begin{equation*}
    \partial S( \alpha, ( \phi_{ i_1 }, \dots, \phi_{ i_n } ) )
    =
    \Omega_{ \tau }( \partial S ).
\end{equation*}
Combining the equalities establishes the claim.
\end{proof}

\begin{lemma}\label{lemm:consistency:Pankka}
Let $M$ be an $n$-dimensional Riemannian manifold, with or without boundary, and $f \in W^{1,n}_{loc}( M, N )$. If $\omega$ is a bounded and smooth differential $n$-form on a complete Riemannian manifold $N$, with or without boundary, then
\begin{equation*}
    f^{\star}\Omega_\omega( T )
    =
    \int_{ E_T } \langle f^{\star}\omega, v_{T} \rangle \theta_{T} \,d\mathcal{H}^n
    \quad\text{for every $T \in \mathbf{I}_{n,c}( \rho_0 )$.}
\end{equation*}
\end{lemma}
\begin{proof}
We isometrically embed $M$ into $\ell^{\infty}$ for the duration of the following proof. The embedding is only applied implicitly when using the results from \cite{Amb:Kir:00}.

The notation $f^{\star}\omega$ has several equivalent definitions based on various definitions of weak differentials of $f$. We consider two.
\begin{enumerate}
    \item[(1)] We Nash embed $N$ into an Euclidean space and use the usual Euclidean-valued weak differential of Sobolev mappings;
    \item[(2)] We isometrically embed $N$ into $\ell^{\infty}$ and consider a Lusin--Lipschitz decomposition $( B_i, f_i )_{ i = 0 }^{ \infty }$ of $f$. Then every $f_i$, $i \geq 1$ is $\text{weak}^{*}$-differentiable at $\mathcal{H}^n$-almost every $x \in N$ in the sense \cite[Theorem 8.1]{Amb:Kir:00}; the differential denoted by $d^{M}f_i$. We define $df$ so that it coincides with $d^{M}f_i$ in $B_i$ whenever $d^{M}f_i$ exists, and set the differential to be zero otherwise.
\end{enumerate}
We observe that both of these definitions satisfy the axioms of the intrinsic colocal weak derivative \cite[Definition 1.2]{Con:VanSch:16}, so they coincide $\mathcal{H}^{n}$-almost everywhere by the uniqueness result \cite[Proposition 1.5]{Con:VanSch:16}. Definition (2) is directly linked to \eqref{eq:lusinlipschitz}, which is the definition we apply from this point onwards.

Consider a basis representation of $$\omega = \sum_{ j = 1 }^{ \infty } \psi_j \sum_{ I } \omega_{j,I} ( \varphi^{j,I} )^{\star}( \mathrm{Vol}_{ \mathbb{R}^n } )$$ of a bounded differential $n$-form. Here $\mathrm{Vol}_{ \mathbb{R}^n }$ is the Riemannian volume form on $\mathbb{R}^n$, $\varphi^{j} \colon U_{j} \to [0, \infty) \times \mathbb{R}^{ \mathrm{dim}(N) - 1 }$ is a smooth chart, with bounded differential, and $\varphi^{ j, I }$ is the mapping formed from the components $( \varphi^{ j }_{ i_1 }, \dots, \varphi^{ l }_{ i_n } )$ corresponding to the increasing multi-index $I = ( i_1, \dots, i_n )$, for $i_{j} \in \left\{1, \dots, \mathrm{dim}(N)\right\}$. Here the $\psi_j$ are chosen in such a way that $U_j$ contains the support of $\psi_j$ and that $\left\{ \psi_j \right\}_{ j = 1 }^{ \infty }$ is a smooth partition of unity of $N$. We assume that the partition of unity has infinite number of terms since the finite case is simpler and follows with minor modifications.

We consider componentwise McShane extensions of $\psi_j$, $\omega_{j,I}$ and $\varphi^j$ to $\ell^{\infty}$, and use the same notation for the extensions. We define
\begin{equation*}
    f^{\star}( \chi_E \omega )
    =
    \sum_{ j = 1 }^{ \infty } 
    \sum_{ I } 
        \left( \left( \chi_E \psi_j \omega_{j,I} \right) \circ f \right) f^{\star}\left( ( \varphi^{j,I} )^{\star}( \mathrm{Vol}_{ \mathbb{R}^n } ) \right),
    \quad\text{for every Borel $E \subset N$,}
\end{equation*}
by pointwise evaluation using the usual linear algebra pullback, see e.g. \cite[Chapter 1]{Fed:69}, and the weak differential of $f$.

We fix $T \in \mathbf{I}_{n,c}( \rho_0 )$, and observe that the measure
\begin{equation*}
    \mu( E )
    \coloneqq
    \sup_{ x \in N } \|\omega\|(x)
    \int_{ f^{-1}(E) } |df|^{n}(x) \,d\|T\|(x),
    \quad E \subset N
\end{equation*}
is a finite Radon measure. Hence, given $\epsilon > 0$, there is a compact set $E \subset N$ for which $\mu( N \setminus E ) < \epsilon/2$. We deduce from \eqref{eq:mass-energy} and the inequality $| f^{\star}\omega | \leq |df|^n \sup_{ x \in N } \|\omega\|(x)$ that
\begin{align*}
    | \Omega_{ \omega }( f_\star T  ) -  \int_{ E_T } \langle f^{\star} \omega, \tau_T \rangle \theta_T \,d\mathcal{H}^n |
    <
    \epsilon
    +
    \left| \Omega_{ \omega }( (f_\star T) \llcorner E ) - \int_{ E_T \cap f^{-1}( E ) } \langle f^{\star} \omega, \tau_T \rangle \theta_T \,d\mathcal{H}^n \right|.
\end{align*}
The claim follows by setting $\epsilon \rightarrow 0^{+}$ after we show that the second term in the upper bound is always zero for compact $E \subset N$.

Let $J$ denote the indices for which $\psi_{j}|_{E}$ is not identically zero. Observe that $J$ is finite. Then the identity \eqref{eq:lusinlipschitz} implies
\begin{equation*}
    \Omega_{ \omega }( f_\star T \llcorner E )
    =
    \sum_{ j \in J }
    \sum_{ I }
    \sum_{ i = 1 }^{ \infty }
    ( (f_i)_\sharp T\llcorner B_i )( \chi_{E}\psi_j \omega_{j,I}, \phi^{j,I} ).
\end{equation*}
The definition of the pushforward under $f_i$, the identity \eqref{eq:representation:current} and the chain rule for the differential $d^{E_T}( \varphi^{j,I} \circ f_i )$ show
\begin{equation*}
    \sum_{ i = 1 }^{ \infty }
    ( (f_i)_\sharp T\llcorner B_i )( \chi_{E}\psi_j \omega_{j,I}, \phi^{j,I} )
    =
    \int_{ E_T }
        \langle f^{\star}( \chi_{E}\psi_j \omega_{j,I} ( \varphi^{j,I} )^{\star}( \mathrm{Vol}_{ \mathbb{R}^n } ) ), \tau_T \rangle
        \theta_{T}
    \,d\mathcal{H}^n.
\end{equation*}
We conclude that
\begin{equation*}
    \Omega_{ \omega }( f_\star T \llcorner E )
    =
    \int_{ E_T } \langle f^{\star}( \chi_{E}\omega ), \tau_T \rangle \theta_T \,d\mathcal{H}^n.
\end{equation*}
This suffices for the claim.
\end{proof}

\section{Quasiregular curves on singular spaces}\label{sec:singularsetting}
In this section, we consider a complete metric space $N$ and an $n$-dimensional oriented Riemannian manifold $M$, with or without boundary, for $n \geq 2$. Top-dimensional compact Lipschitz submanifolds $\widetilde{M}$ of $M$ are endowed with the orientation induced from $M$, considered as integral currents $[\widetilde{M}]$.

\subsection{Isoperimetric inequalities of Euclidean type}\label{sec:isoperim}
In this section we consider a complete metric space $N$ satisfying various isoperimetric inequalities. For this purpose, given $j = 1,2,\dots$ and $T \in \mathbf{I}_{j}( N )$, we denote
\begin{equation*}
    \mathrm{Fillvol}( T )
    \coloneqq
    \inf\left\{ M(S) \colon S \in \mathbf{I}_{j+1}(X) \,\,\text{and}\,\, \partial S = T \right\}.
\end{equation*}
We say that $N$ supports an \emph{isoperimetric inequality of Euclidean type with dimension $j$, mass bound $M$ and constant $A$} if
\begin{equation}\label{eq:smallmass}
    \text{$T \in \mathbf{I}_{j}( N )$, $M( T ) \leq M$, and $\partial T = 0$}
    \quad\text{imply}\quad
    \mathrm{Fillvol}( T )
    \leq
    A( M(T) )^{ \frac{j+1}{j} }.
\end{equation}
When $M = \infty$, we say that $N$ supports an \emph{isoperimetric inequality of Euclidean type with dimension $j$ and constant $A$}.

We refer the reader to \cite{Wen:07,Bas:Wen:You:21} and references therein for further information about spaces satisfying \eqref{eq:smallmass}. Our main applications of interest are metric spaces $N$ for which there is $\delta, \gamma > 0$ such that subsets of diameter at most $\delta$ are $\gamma$-Lipschitz contractible (see \cite[Section 3.2]{Wen:07} for the precise definition). These include the classes of compact Lipschitz manifolds, Banach spaces, and CAT($\kappa$)-spaces.

Almgren proved, cf. \cite[Theorem 10]{Alm:86}, that when $N = \mathbb{R}^m$ for any $m \geq 1$ and any $j \geq 1$, the optimal Euclidean isoperimetry constants are
\begin{equation}\label{eq:euclidean constant}
    M = \infty \quad\text{and}\quad A_{ j } \coloneqq \omega_{j+1}^{ \frac{ 1 }{ j } } (j+1)^{ \frac{j+1}{j} },
\end{equation}
where $\omega_{j+1}$ is the volume (Lebesgue measure) of the $(j+1)$-dimensional unit ball $\mathbb{B}^{j+1}$ in $\mathbb{R}^{j+1}$.

We prove \Cref{cor:Sob:isoperimetry} after the following lemma.
\begin{lemma}\label{lemm:homology:to:filling}
When $N$ supports an isoperimetric inequality of Euclidean type with dimension $n$, mass bound $M'$ and constant $A'$, and $P \in \mathbf{I}_{n}(N)$ satisfies $M(P) \leq M'/2$, then
\begin{equation*}
    \inf_{ S \in \mathbf{I}_{n+1}(N) }
    M( P + \partial S )
    =
    \mathrm{Fillvol}( \partial P ).
\end{equation*}
In case $N$ supports an isoperimetric inequality of Euclidean type with dimension $n-1$, mass bound $M$ and constant $A$, then
\begin{equation*}
    P \in \mathbf{I}_{n}(N)
    \quad\text{and}\quad
    M( P ) \leq A M^{ \frac{n}{n-1} }
    \quad\text{imply}\quad
    \mathrm{Fillvol}( \partial P )
    \leq
    A \left( M( \partial P ) \right)^{ \frac{n}{n-1} }.
\end{equation*}
\end{lemma}
\begin{proof}
Given the isoperimetric inequality of dimension $n$, the equality $\partial T = \partial P$ and $M( T ) \leq M( P ) \leq M'/2$ imply $P-T = \partial S$ for some $S \in \mathbf{I}_{n+1}(N)$. Hence
\begin{equation*}
    \mathrm{Fillvol}( \partial P )
    =
    \inf_{ S \in \mathbf{I}_{n+1}(N) } M( P + \partial S )
    \quad\text{whenever $M(P) \leq M'/2$.}
\end{equation*}
Next, suppose that $N$ supports an isoperimetric inequality of Euclidean type with dimension $n-1$, mass bound $M$ and constant $A$. If $M( P ) \leq A M^{ \frac{n}{n-1} }$, then
\begin{equation*}
    \mathrm{Fillvol}( \partial P )
    \leq
    A\left( M( \partial P ) \right)^{ \frac{n}{n-1} }.
\end{equation*}
Indeed, otherwise
\begin{equation*}
    \mathrm{Fillvol}( \partial P )
    >
    A\left( M( \partial P ) \right)^{ \frac{n}{n-1} },
\end{equation*}
so $M( \partial P ) < M$. These two inequalities then contradict the isoperimetric inequality.
\end{proof}
\begin{corollary}\label{cor:localisoperimetry}
Suppose that $N$ supports an isoperimetric inequality of Euclidean type with dimension $n$, mass bound $M'$ and constant $A'$ and with dimension $(n-1)$, mass bound $M$, and constant $A$, respectively, for some $n \geq 2$.

Suppose that $M$ is an oriented $n$-dimensional Riemannian manifold, with or without boundary, and $f \in W^{1,n}_{loc}( M, N )$ satisfies
\begin{equation*}
    \int_{ M } |df|^{n} \,d\mathcal{H}^n
    \leq
    \min\left\{
        A M^{ \frac{n}{n-1} },
        \frac{ M' }{ 2 }
    \right\}.
\end{equation*}
Then every closed and additive $\Omega \in \mathcal{L}_\infty( \mathbf{I}_{n}(N) )$ with constant upper norm $x \mapsto C$ satisfies
\begin{align*}
    | \Omega( f_\star[\widetilde{M}] ) |
    \leq
    C
    \inf_{ S \in \mathbf{I}_{n+1}(N) }
    M( f_\star[\widetilde{M}] + \partial S )
    \leq
    A C
    \left( \int_{ \partial \widetilde{M} } |df|^{n-1} \,d\mathcal{H}^{n-1} \right)^{ \frac{n}{n-1} }
\end{align*}
whenever $[ \widetilde{M} ] \in \mathbf{I}_{n,c}( \rho_0 )$.
\end{corollary}
\Cref{cor:Sob:isoperimetry} is an immediate consequence of \Cref{cor:localisoperimetry}.
\begin{proof}[Proof of \Cref{cor:localisoperimetry}]
Consider a Borel $\rho_0 \in \mathcal{L}^{n}_{+}( M )$ such that $f$ satisfies the conclusion of \Cref{thm:pushforward} in $\mathbf{I}_{n,c}( \rho_0 )$. The upper bound on energy combined with~\eqref{eq:mass-energy} show that $f_\star[ \widetilde{M} ]$ satisfies both of the mass upper bounds of \Cref{lemm:homology:to:filling}. Then the claim follows from \Cref{lemm:homology:to:filling}, the closedness of $\Omega$ and $\Omega$ having the upper norm $x \mapsto C$.
\end{proof}

\subsection{Calibrations and quasiregular curves}
We consider an additive cochain
\begin{equation*}
    \Omega \colon \mathbf{I}_{n}(N) \to \mathbb{R},
\end{equation*}
for $n \in \left\{0,1,\dots\right\}$. Before giving the definition of quasiregular curves, we prove show that cochains with bounded upper norms have $\mathcal{H}^{n}$-essentially unique \emph{minimal upper norms}. To this end, we observe the following.
\begin{lemma}\label{lemm:strongnorm}
Suppose that a cochain $\Omega \colon \mathbf{I}_{n}( N ) \to \mathbb{R}$ has $\mathcal{H}^{n}$-essentially bounded upper norms $g_1$ and $g_2$. Then $g = \min\left\{ g_1, g_2 \right\}$ is an upper norm of $\Omega$.
\end{lemma}
\Cref{lemm:strongnorm} yields the following lattice property.
\begin{corollary}\label{cor:lattice}
Suppose that a cochain $\Omega \colon \mathbf{I}_{n}( N ) \to \mathbb{R}$ has an $\mathcal{H}^n$-essentially bounded norm. Then the $\mathcal{H}^{n}$-essentially bounded upper norms of $\Omega$ is closed with respect to taking Borel $\mathcal{L}^{\infty}( N; \mathcal{H}^n )$-convex combinations and taking pointwise $\mathcal{H}^{n}$-essential infima.
\end{corollary}
Given \Cref{cor:lattice}, we deduce that the collection of $\mathcal{H}^{n}$-essentially bounded upper norms of $\Omega$ is empty or there is a \emph{minimal upper norm}, unique up to $\mathcal{H}^n$-negligible sets. The minimal upper norm is denoted by $\|\Omega\|$. The minimality should be understood as follows: every $\mathcal{H}^{n}$-essentially bounded upper norm of $\Omega$ is bounded from below by $\|\Omega\|$ $\mathcal{H}^n$-almost everywhere.

\begin{proof}[Proof of \Cref{lemm:strongnorm}]
Let $E = \left\{ g_1 \leq g_2 \right\}$ and $T \in \mathbf{I}_{n}( N )$. Given $\epsilon > 0$, the $\mathcal{H}^{n}$-essential boundedness of $g_1$ and $g_2$ yield the existence of a compact set $K \subset E$ and open set $V \supset E$ for which
\begin{equation*}
    \int_{ V \setminus K } g_1 + g_2 \,d\|T\| < \frac{ \epsilon }{ 2 };
\end{equation*}
here we apply the fact that $\| T \|$ is absolutely continuous with respect to $\mathcal{H}^n$ and the fact that $\|T\|$ is Radon. Consider the Lipschitz function $h(z) = d( K, z )$, $z \in N$. Then, by the localization lemma \cite[Lemma 5.3]{AK:00:current}, $\mathcal{L}^{1}$-almost every $t \in \mathbb{R}$ is such that
\begin{equation*}
    T_- \coloneqq T\llcorner{ \left\{ h \leq t \right\}  } \in \mathbf{N}_{n}( N ),
    \quad
    T_+ \coloneqq T\llcorner{ \left\{ h > t \right\}  } \in \mathbf{N}_{n}( N ),
    \quad\text{and}\quad
    T = T_- + T_+.
\end{equation*}
In fact, here $T_{\pm} \in \mathbf{I}_{n}( N )$ by the closure theorem \cite[Theorem 8.5]{AK:00:current} and the construction of $T_+$ from \cite[Lemma 5.3]{AK:00:current}. This can alternatively be deduced from the intrinsic representation theorem of integer-rectifiable currents, cf. \cite[Theorem 9.1]{AK:00:current}.

We consider only those $0 < t$ for which $\left\{ h \leq t \right\} \subset V$. For such $t$, the upper norm inequality gives
\begin{align*}
    | \Omega( T_- ) |
    &\leq
    \int_{ N } g_1 \,d\|T_-\|
    =
    \int_{ \left\{ h \leq t \right\} } g_1 \,d\|T\|
    \leq
    \int_{ E } g_1 \,d\|T\|
    +
    \frac{ \epsilon }{ 2 }.
\end{align*}
By symmetry in the argument and the subadditivity of $\Omega$, we deduce that
\begin{align*}
    | \Omega( T ) |
    \leq
    | \Omega( T_- ) |
    +
    | \Omega( T_{+} ) |
    \leq
    \int_{E} g_1 \,d\|T\|
    +
    \int_{N \setminus E} g_2 \,d\|T\|
    +
    \epsilon.
\end{align*}
Since $\epsilon > 0$ was arbitrarily small and $T \in \mathbf{I}_{n}( N )$ arbitrary, the claim follows.
\end{proof}

\begin{definition}\label{def:QR-curve}
Consider a complete metric space $N$ and a closed $\Omega \in \mathcal{L}_{\infty}( \mathbf{I}_{n}(N) )$.

Given an oriented $n$-dimensional Riemannian manifold $M$, with or without boundary, a mapping $f \colon M \to N$ is a \emph{$K$-quasiregular $\Omega$-curve} if $f \in W^{1,n}_{loc}( M, N )$ and there is a Borel $\rho_0 \in \mathcal{L}^{n}_+( M )$ with
\begin{equation}\label{eq:distortioninequality:copy}
    \int_{ \widetilde{M} }
        \|\Omega\| \circ f
        | df |^n
    \,d\mathcal{H}^n
    \leq
    K \Omega( f_\star [ \widetilde{M} ] )
    \quad\text{for $[\widetilde{M}] \in \mathbf{I}_{n,c}( \rho_0 )$.}
\end{equation}
Given constants $0 < c < C$, the pair $( \Omega, f )$ is \emph{$(C,c)$-calibrated} if $x \mapsto C$ is an upper norm of $\Omega$ and
\begin{equation}\label{eq:boundedfrombelow}
    c |df|^n
    \leq
    ( \|\Omega\| \circ f ) |df|^n
    \quad\text{$\mathcal{H}^n$-almost everywhere.}
\end{equation}
\end{definition}

\begin{remark}[Consistency]\label{rem:cons}
When $N$ is a complete Riemannian manifold and $\omega$ a closed and bounded differential form, the form induces a closed cochain $\Omega_\omega$, cf. \Cref{lemm:local:to:global}, in such a way that the comass of $\omega$ is a minimal upper norm of $\Omega_\omega$. Moreover, \Cref{lemm:consistency:Pankka} shows that the pullback $f^{\star}\Omega_\omega$ is induced by the pullback $f^{\star}\omega$. Now if we add the assumption that $f$ is continuous and $\|\omega\|(x) > 0$ everywhere in $N$, the definition above coincides with the one by Pankka \cite{Pan:20}.
\end{remark}

 The lower bound \eqref{eq:boundedfrombelow} is motivated by the calibrated geometries of Harvey and Lawson in \cite{Har:Law:82}. The connection is highlighted by the following proposition, the proof of which is immediate from the definition above.
\begin{proposition}\label{prop:quasiminimality}
Let $M$ be an $n$-dimensional Riemannian manifold, with or without boundary, $N$ a complete metric space and $f \colon M \to N$ a $K$-quasiregular $\Omega$-curve with $( \Omega, f )$ being $( C, c )$-calibrated. Then
\begin{equation}\label{eq:quasiminimality}
    cM( f_\star[ \widetilde{M} ] )
    \leq
    c\int_{ \widetilde{M} } |df|^n \,d\mathcal{H}^n
    \leq
    K \Omega( f_\star[ \widetilde{M} ] ),
    \quad\text{$[\widetilde{M}] \in \mathbf{I}_{n,c}( \rho_0 )$}.
\end{equation}
\end{proposition}
Observe that the smallest and largest terms in \Cref{prop:quasiminimality} imply that $f_\star[ \widetilde{M} ]$ is a \emph{quasiminimizer} of mass in its homology class of integral currents. That is,
\begin{equation*}
    M( f_\star[ \widetilde{M} ] )
    \leq
    Q \inf_{ S \in \mathbf{I}_{n+1}( N ) } M( f_\star[ \widetilde{M} ] + \partial S ),
    \quad\text{for $Q = KC/c$.}
\end{equation*}
This observation is connected to the quasiminimality of $f$ in the sense of Gromov \cite[Definition 6.36]{Gro:07}. The latter inequality in \eqref{eq:quasiminimality}, combined with the mass--energy inequality \eqref{eq:mass-energy}, implies quasiminimality of $f$ in \emph{energy} for Sobolev mappings in the same homology class, cf. \cite[Section 1.1]{Pan:20}. See \Cref{thm:mainresult:2} below for a related result.

The observations above relate quasiregular curves to the area- and energy-minimization problems considered by Lytchak and Wenger in \cite{Lyt:Wen:16,Lyt:Wen:17:energyarea,Lyt:Wen:18:CAT,Lyt:Wen:20}. See also \cite{Gia:Giu:84,Guo:Xia:19} and references therein.

\subsection{Regularity of quasiregular curves}

The following theorem is a simple consequence of the tools developed so far, formulated for quasiregular curves defined on $\mathbb{B}^n$.
\begin{theorem}\label{thm:mainresult:1}
Suppose that $N$ supports an isoperimetric inequality of Euclidean type with dimension $n$, mass bound $M'$ and constant $A'$ and with dimension $(n-1)$, mass bound $M$, and constant $A$, respectively, for some $n \geq 2$.

If $f \colon \mathbb{B}^{n} \to N$ is a $K$-quasiregular $\Omega$-curve, $( \Omega, f )$ is $( C, c )$-calibrated, and
\begin{equation}\label{eq:massupperbound:thm}
    \int_{ \mathbb{B}^n } |df|^n \,d\mathcal{H}^n
    \leq
    \min\left\{
        M^{ \frac{n}{n-1} } A,
        \frac{ M' }{ 2 }
    \right\},
\end{equation}
then $f$ has a continuous representative $\widehat{f}$ such that
\begin{equation*}
    d( \widehat{f}( x ), \widehat{f}( y ) )
    \leq
    C
    \left( \int_{ \mathbb{B}^n } |df|^n(x) \,d\mathcal{H}^n \right)^{1/n}
    | x - y |^{\beta},
    \quad
    x,y \in B\left( 0, 1/2 \right),
\end{equation*}
for $Q = KC/c$, $\beta= \left( Q A / A_{n-1} \right)^{-1}$ and $C = C(n,\beta) > 0$.
\end{theorem}
By conformal invariance of the left-hand side of \eqref{eq:massupperbound:thm}, \Cref{thm:mainresult:1} readily implies a local version of the statement and can be applied e.g. on Riemannian manifolds using normal coordinates. Observe, however, that passing to charts typically increases the constant $K$. The conformal invariance also implies \Cref{thm:yes} as an immediate corollary.

\begin{proof}[Proof of \Cref{thm:mainresult:1}]
We fix any Borel $\rho_0 \in \mathcal{L}^{n}_{+}( \mathbb{B}^n )$ for which $f$ satisfies the conclusions of \Cref{thm:pushforward} and \Cref{prop:quasiminimality} whenever $[ \widetilde{M} ] \in \mathbf{I}_{n,c}( \rho_0 )$.

Observe that $f$ satisfies the assumptions of \Cref{cor:localisoperimetry}. Next, we apply \eqref{eq:quasiminimality} together with the conclusion of \Cref{cor:localisoperimetry}. We deduce that
\begin{equation}\label{eq:maintool:pre}
    \int_{ \widetilde{M} } |df|^n \,d\mathcal{H}^n
    \leq
    A Q \left( \int_{ \partial \widetilde{M} } |df|^{n-1} \,d\mathcal{H}^{n-1} \right)^{ \frac{n}{n-1} }
    \quad\text{whenever $[\widetilde{M}] \in \mathbf{I}_{n,c}( \rho_0 )$.}
\end{equation}
Fix any $x \in B( 0, 1/2 )$. For $0 < r \leq 1/2$, consider the functions
\begin{align*}
    H(r) &\coloneqq \int_{ B( x, r ) } |df|^n \,d\mathcal{H}^n,
    \quad
    &&h(r) \coloneqq \int_{ \partial B( x, r ) } |df|^{n-1} \,d\mathcal{H}^{n-1}
    \\
    c(r) &\coloneqq \int_{ \partial B( x, r )  } |df|^{n} \,d\mathcal{H}^{n-1}
    \quad\text{and}\quad
    &&s(r)\coloneqq
    \left( \mathcal{H}^{n-1}( \partial B( x, r ) ) \right)^{\frac{ 1}{ n-1 }}.
\end{align*}
By Fubini's theorem, \eqref{eq:maintool:pre} holds for $B(x,r)$ for almost every $0 < r \leq 1/2$. In other words,
\begin{equation}\label{eq:maintool}
    H( r )
    \leq
    Q
    A
    \left( h( r ) \right)^{ \frac{n}{n-1} }
    \quad\text{for almost every $0 < r \leq 1/2$.}
\end{equation}
Moreover, by Hölder's inequality,
\begin{align}\label{eq:holder}
    h(r)
    \leq
    \left( c(r)  s(r) \right)^{ \frac{n-1}{n} }
    \quad\text{for every $0 < r \leq 1/2$.}
\end{align}
Notice that
\begin{equation}\label{eq:essentialinequality}
    s(r)
    =
    \left( n \omega_n \right)^{ \frac{1}{n-1} }
    r
    =
    \frac{ 1 }{ n A_{n-1} } r. 
\end{equation}
We combine \eqref{eq:maintool}, \eqref{eq:holder}, and \eqref{eq:essentialinequality} to deduce that
\begin{equation*}
    H(r)
    \leq
    Q
    A c(r) s(r)
    =
    \frac{ 1 }{ n \beta }
    r c(r)
    \quad\text{for almost every $0 < r \leq 1/2$}.
\end{equation*}
Observe that $r \mapsto H(r)$ is absolutely continuous and $H'(r) = c(r)$ for almost every $0 < r \leq 1/2$ by Fubini's theorem. Consequently,
\begin{equation}\label{eq:key:PDE}
    H(r)
    \leq
    \frac{ 1 }{ n \beta }
    r
    H'(r)
    \quad\text{for almost every $0 < r \leq 1/2$.}
\end{equation}
Grönwall's lemma and \eqref{eq:key:PDE} imply
\begin{equation}\label{eq:key:exponent}
    H(s)
    \leq
    H( S )
    \left( \frac{ s }{ S } \right)^{ n \beta }
    \quad\text{for every $0 < s < S \leq 1/2$}.
\end{equation}

Next, we denote
\begin{equation*}
    K(x, S)
    \coloneqq
    \left(
        \omega_n
    \right)^{ \frac{n-1}{n} }
    \left(
        H( S )
    \right)^{1/n}
   S^{ - \beta }.
\end{equation*}
Hölder inequality and \eqref{eq:key:exponent} show
\begin{align}
\label{eq:inequality}
    \| |df| \|_{ L^{1}( B(x,s) ) }
    \leq
    \left( \omega_n s^{n} \right)^{ \frac{n-1}{n} }
    \left( H( S ) \left( \frac{ s }{ S } \right)^{ n \beta } \right)^{ \frac{1}{n} }
    =
    K( x, S )
    s^{ (n-1) + \beta }.
\end{align}
We set $S = 1/2$ and notice that
\begin{equation*}
    C\left( 0, 1/2 \right)
    \coloneqq
    ( \omega_{n} )^{ \frac{n-1}{n} }
    \left( \int_{ \mathbb{B}^{n} } |df|^{n} \,d\mathcal{H}^n \right)^{ \frac{1}{n} }
    \left( 1/2 \right)^{ - \beta }
    \geq
    K\left( x, 1/2 \right),
    \quad x \in B\left( 0, 1/2 \right).
\end{equation*}
It is a classical result due to Morrey, see, e.g. \cite[Theorem 7.19, p 165]{Gil:Tru:01}, that~\eqref{eq:inequality} and the uniform upper bound on $K(x, 1/2 )$ imply
\begin{equation*}
    \mathrm{esssup}_{ y, z \in B( 0, 1/2 ), y \neq z }
    \frac{ d(\widehat{f}(y), \widehat{f}( z )) }{ |y-z|^\beta }
    \leq
    C\left( n, \beta \right)
    C\left( 0, 1/2 \right).
\end{equation*}
In fact, the quoted result is formulated only for real-valued functions but considering $N$ as a subset of $L^{\infty}(N)$ and recalling that the image of $f$ is essentially separable, the result readily generalizes. The claim about the existence of a continuous representative of $f$ follows by covering $\mathbb{B}^{n}$ by balls such that the ball with double the radius is contained in $\mathbb{B}^n$ and by doing the obvious modifications in the argument above.
\end{proof}
Next, we show that quasiregular curves are (local) quasiminimizers of energy in the following sense.

\begin{theorem}\label{thm:mainresult:2}
Suppose that $N$ supports an isoperimetric inequality of Euclidean type with dimension $n$, mass bound $M'$ and constant $A'$, for some $n \geq 1$.

If $f \colon \mathbb{B}^{n} \to N$ is a $K$-quasiregular $\Omega$-curve, $( \Omega, f )$ is $( C, c )$-calibrated, and
\begin{equation}\label{eq:massupperbound:thm:new}
    \int_{ \mathbb{B}^n } |df|^n \,d\mathcal{H}^n
    \leq
    \frac{ M' }{ 2 },
\end{equation}
then whenever $g \in W^{1,n}_{loc}( \mathbb{B}^{n}, N )$ is such that $f = g$ $\mathcal{H}^n$-almost everywhere outside a compact set $E \subset \mathbb{B}^n$, the inequality
\begin{equation*}
    \int_{ E } |df|^n \,d\mathcal{H}^n
    \leq
    Q
    \int_{ E } |dg|^n \,d\mathcal{H}^n
\end{equation*}
holds for $Q = KC / c$.
\end{theorem}
\Cref{cor:regularity:lip+harm} is an immediate consequence of Theorems \ref{thm:mainresult:1} and \ref{thm:mainresult:2}. Indeed, the modulus of continuity follows from \Cref{thm:mainresult:1}, while $f$ being $n$-harmonic follows from \Cref{thm:mainresult:2}.

\begin{proof}[Proof of \Cref{thm:mainresult:2}]
There is a Borel $\rho_0 \in \mathcal{L}^{n}_+( \mathbb{B}^n )$ so that the pushforwards $f_\star$ and $g_\star$ satisfy the conclusions of \Cref{thm:pushforward}, $f$ the conclusions of \Cref{prop:quasiminimality} on $\mathbf{I}_{n,c}( \rho_0 )$, and that
\begin{equation}\label{eq:failureset}
    \left\{ f \neq g \right\} \subset \left\{ \rho_0 = \infty \right\}.
\end{equation}
Given any compact Lipschitz $n$-dimensional manifold $\widetilde{M} \subset \mathbb{B}^{n}$ with $\partial \widetilde{M} \subset \mathbb{B}^{n} \setminus E$ and $[\widetilde{M}] \in \mathbf{I}_{n,c}( \rho_0 )$, we deduce from \eqref{eq:lusinlipschitz} and \eqref{eq:failureset} that $f_\star[ \partial \widetilde{M} ] = g_\star[ \partial \widetilde{M} ]$.

Observe that the energy upper bound on $f$ implies that $f_\star[ \widetilde{M} ]$ satisfies the first upper bound appearing in \Cref{lemm:homology:to:filling}. Therefore $f_\star[ \partial \widetilde{M} ] = g_\star[ \partial \widetilde{M} ]$ implies
\begin{equation*}
    \inf_{ S \in \mathbf{I}_{n+1}(M) } M( f_\star[ \partial \widetilde{M} ] + \partial S )
    =
    \mathrm{Fillvol}( f_\star[ \partial \widetilde{M} ] )
    \leq
    M( g_\star[ \widetilde{M} ] )
    \leq
    \int_{ \widetilde{M} } |dg|^{n} \,d\mathcal{H}^n.
\end{equation*}
Combining these inequalities with \eqref{eq:quasiminimality} yields
\begin{equation}\label{eq:mainconclusion}
    \int_{ \widetilde{M} } |df|^{n} \,d\mathcal{H}^n
    \leq
    Q \int_{ \widetilde{M} } |dg|^{n} \,d\mathcal{H}^n.
\end{equation}

We recall that $\mathbb{B}^{n} \setminus E$ can be exhausted by open sets $U_j \subset U_{j+1}$ so that $\partial U_j \cap \partial \mathbb{B}^{n} = \partial \mathbb{B}^{n}$ and $\partial U_j \setminus \partial \mathbb{B}^{n}$ is a smooth $(n-1)$-dimensional submanifold of $\mathbb{B}^{n} \setminus E$ contained in
\begin{equation*}
    E_{j} \coloneqq \left\{ x \in \mathbb{B}^{n} \colon d( E, x ) < 2^{-j} \right\},
    \quad\text{cf. \cite[Proposition 8.2.1]{Dan:08}.}
\end{equation*}
Up to a perturbation of $\partial U_j \setminus \partial \mathbb{B}^{n}$ in $E_j \setminus E$ using a tubular neighbourhood of $\partial U_j \setminus \partial \mathbb{B}^{n}$, we may assume that $\widetilde{M}_j \coloneqq \mathbb{B}^{n} \setminus U_j$ induces an integral current in $\mathbf{I}_{n,c}( \rho_0 )$. Then applying \eqref{eq:mainconclusion} for each $j$ and passing to the limit $j \rightarrow \infty$ shows the claim.
\end{proof}

\section*{Acknowledgements}
We thank Eero Hakavuori, Susanna Heikkilä, Pekka Pankka and Kai Rajala for comments on an earlier version of this manuscript. We also thank Jonathan Pim and Athanasios Tsantaris for discussions regarding quasiregular curves.

\bibliographystyle{alpha}

\begin{thebibliography}{{Kan}21}

\bibitem[AIM09]{Ast:Iwa:Mar:09}
Kari Astala, Tadeusz Iwaniec, and Gaven Martin.
\newblock {\em Elliptic partial differential equations and quasiconformal
  mappings in the plane}, volume~48 of {\em Princeton Mathematical Series}.
\newblock Princeton University Press, Princeton, NJ, 2009.

\bibitem[AK00a]{AK:00:current}
Luigi Ambrosio and Bernd Kirchheim.
\newblock Currents in metric spaces.
\newblock {\em Acta Math.}, 185(1):1--80, 2000.

\bibitem[AK00b]{Amb:Kir:00}
Luigi Ambrosio and Bernd Kirchheim.
\newblock Rectifiable sets in metric and {B}anach spaces.
\newblock {\em Math. Ann.}, 318(3):527--555, 2000.

\bibitem[Alm86]{Alm:86}
F.~Almgren.
\newblock Optimal isoperimetric inequalities.
\newblock {\em Indiana Univ. Math. J.}, 35(3):451--547, 1986.

\bibitem[Bet91]{Bet:91}
Fabrice Bethuel.
\newblock The approximation problem for {S}obolev maps between two manifolds.
\newblock {\em Acta Math.}, 167(3-4):153--206, 1991.

\bibitem[Bog07]{Bog:07:II}
V.~I. Bogachev.
\newblock {\em Measure theory. {V}ol. {I}, {II}}.
\newblock Springer-Verlag, Berlin, 2007.

\bibitem[BPVS17]{Bou:Pon:VanSc:17}
Pierre Bousquet, Augusto~C. Ponce, and Jean Van~Schaftingen.
\newblock Density of bounded maps in {S}obolev spaces into complete manifolds.
\newblock {\em Ann. Mat. Pura Appl. (4)}, 196(6):2261--2301, 2017.

\bibitem[BWY21]{Bas:Wen:You:21}
Giuliano {Basso}, Stefan {Wenger}, and Robert {Young}.
\newblock {Undistorted fillings in subsets of metric spaces}.
\newblock {\em arXiv e-prints}, page arXiv:2112.11905, December 2021.

\bibitem[CKM20]{Che:Kari:Mad:20}
Da~Rong Cheng, Spiro Karigiannis, and Jesse Madnick.
\newblock Bubble tree convergence of conformally cross product preserving maps.
\newblock {\em Asian J. Math.}, 24(6):903--984, 2020.

\bibitem[CVS16]{Con:VanSch:16}
Alexandra Convent and Jean Van~Schaftingen.
\newblock Intrinsic co-local weak derivatives and {S}obolev spaces between
  manifolds.
\newblock {\em Ann. Sc. Norm. Super. Pisa Cl. Sci. (5)}, 16(1):97--128, 2016.

\bibitem[Dan08]{Dan:08}
Daniel Daners.
\newblock Domain perturbation for linear and semi-linear boundary value
  problems.
\newblock In {\em Handbook of differential equations: stationary partial
  differential equations. {V}ol. {VI}}, Handb. Differ. Equ., pages 1--81.
  Elsevier/North-Holland, Amsterdam, 2008.

\bibitem[Don96]{Don:96}
S.~K. Donaldson.
\newblock Symplectic submanifolds and almost-complex geometry.
\newblock {\em J. Differential Geom.}, 44(4):666--705, 1996.

\bibitem[DS89]{Don:Sul:89}
S.~K. Donaldson and D.~P. Sullivan.
\newblock Quasiconformal {$4$}-manifolds.
\newblock {\em Acta Math.}, 163(3-4):181--252, 1989.

\bibitem[EBS21]{EB:Sou:21}
Sylvester Eriksson-Bique and Elefterios {Soultanis}.
\newblock {Curvewise characterizations of minimal upper gradients and the
  construction of a Sobolev differential}.
\newblock {\em arXiv e-prints}, page arXiv:2102.08097, February 2021.

\bibitem[Fed69]{Fed:69}
Herbert Federer.
\newblock {\em Geometric measure theory}.
\newblock Die Grundlehren der mathematischen Wissenschaften, Band 153.
  Springer-Verlag New York Inc., New York, 1969.

\bibitem[FF60]{Fed:Fle:60}
Herbert Federer and Wendell~H. Fleming.
\newblock Normal and integral currents.
\newblock {\em Ann. of Math. (2)}, 72:458--520, 1960.

\bibitem[Fug57]{Fug:57}
Bent Fuglede.
\newblock Extremal length and functional completion.
\newblock {\em Acta Math.}, 98:171--219, 1957.

\bibitem[GG84]{Gia:Giu:84}
Mariano Giaquinta and Enrico Giusti.
\newblock Quasiminima.
\newblock {\em Ann. Inst. H. Poincar\'{e} Anal. Non Lin\'{e}aire},
  1(2):79--107, 1984.

\bibitem[GHP19]{Gol:Haj:Pak:19}
Pawe{\l} Goldstein, Piotr Haj{\l}asz, and Mohammad~Reza Pakzad.
\newblock Finite distortion {S}obolev mappings between manifolds are
  continuous.
\newblock {\em Int. Math. Res. Not. IMRN}, (14):4370--4391, 2019.

\bibitem[GIZ22]{GB:Iko:Zhu:22}
Miguel {Garc{\'\i}a-Bravo}, Toni {Ikonen}, and Zheng {Zhu}.
\newblock {Extensions and approximations of Banach-valued Sobolev functions}.
\newblock {\em arXiv e-prints}, page arXiv:2208.12594, August 2022.

\bibitem[Gro83]{Gro:83}
Mikhael Gromov.
\newblock Filling {R}iemannian manifolds.
\newblock {\em J. Differential Geom.}, 18(1):1--147, 1983.

\bibitem[Gro85]{Grom:85}
M.~Gromov.
\newblock Pseudo holomorphic curves in symplectic manifolds.
\newblock {\em Invent. Math.}, 82(2):307--347, 1985.

\bibitem[Gro07]{Gro:07}
Misha Gromov.
\newblock {\em Metric structures for {R}iemannian and non-{R}iemannian spaces}.
\newblock Modern Birkh\"{a}user Classics. Birkh\"{a}user Boston, Inc., Boston,
  MA, english edition, 2007.
\newblock Based on the 1981 French original, With appendices by M. Katz, P.
  Pansu and S. Semmes, Translated from the French by Sean Michael Bates.

\bibitem[GT01]{Gil:Tru:01}
David Gilbarg and Neil~S. Trudinger.
\newblock {\em Elliptic partial differential equations of second order}.
\newblock Classics in Mathematics. Springer-Verlag, Berlin, 2001.
\newblock Reprint of the 1998 edition.

\bibitem[GT10]{Gol:Tro:10}
Vladimir Gol'dshtein and Marc Troyanov.
\newblock A conformal de {R}ham complex.
\newblock {\em J. Geom. Anal.}, 20(3):651--669, 2010.

\bibitem[GW79]{Gree:Wu:79}
R.~E. Greene and H.~Wu.
\newblock {$C^{\infty }$} approximations of convex, subharmonic, and
  plurisubharmonic functions.
\newblock {\em Ann. Sci. \'{E}cole Norm. Sup. (4)}, 12(1):47--84, 1979.

\bibitem[GX19]{Guo:Xia:19}
Chang-Yu Guo and Chang-Lin Xiang.
\newblock Regularity of quasi-{$n$}-harmonic mappings into {NPC} spaces.
\newblock {\em Ann. Mat. Pura Appl. (4)}, 198(2):367--380, 2019.

\bibitem[Haj96]{Haj:96}
Piotr Haj{\l}asz.
\newblock Sobolev spaces on an arbitrary metric space.
\newblock {\em Potential Anal.}, 5(4):403--415, 1996.

\bibitem[Haj07]{Haj:07}
Piotr Haj{\l}asz.
\newblock Sobolev mappings: {L}ipschitz density is not a bi-{L}ipschitz
  invariant of the target.
\newblock {\em Geom. Funct. Anal.}, 17(2):435--467, 2007.

\bibitem[Haj09]{Haj:09}
Piotr Haj{\l}asz.
\newblock Density of {L}ipschitz mappings in the class of {S}obolev mappings
  between metric spaces.
\newblock {\em Math. Ann.}, 343(4):801--823, 2009.

\bibitem[Hei23]{Hei:21}
Susanna Heikkil{\"a}.
\newblock Signed quasiregular curves.
\newblock {\em Journal d'Analyse Math{\'e}matique}, Mar 2023.

\bibitem[HIMO08]{Haj:Iwa:Mal:On}
Piotr Haj{\l}asz, Tadeusz Iwaniec, Jan Mal\'{y}, and Jani Onninen.
\newblock Weakly differentiable mappings between manifolds.
\newblock {\em Mem. Amer. Math. Soc.}, 192(899):viii+72, 2008.

\bibitem[HK98]{Hei:Kos:98}
Juha Heinonen and Pekka Koskela.
\newblock Quasiconformal maps in metric spaces with controlled geometry.
\newblock {\em Acta Math.}, 181(1):1--61, 1998.

\bibitem[HKST01]{Hei:Kos:Sha:Ty:01}
Juha Heinonen, Pekka Koskela, Nageswari Shanmugalingam, and Jeremy~T. Tyson.
\newblock Sobolev classes of {B}anach space-valued functions and quasiconformal
  mappings.
\newblock {\em J. Anal. Math.}, 85:87--139, 2001.

\bibitem[HKST15]{Hei:Kos:Sha:Ty:15}
Juha Heinonen, Pekka Koskela, Nageswari Shanmugalingam, and Jeremy~T. Tyson.
\newblock {\em Sobolev spaces on metric measure spaces}, volume~27 of {\em New
  Mathematical Monographs}.
\newblock Cambridge University Press, Cambridge, 2015.
\newblock An approach based on upper gradients.

\bibitem[HL82]{Har:Law:82}
Reese Harvey and H.~Blaine Lawson, Jr.
\newblock Calibrated geometries.
\newblock {\em Acta Math.}, 148:47--157, 1982.

\bibitem[HL03]{Hang:Lin:03}
Fengbo Hang and Fanghua Lin.
\newblock Topology of {S}obolev mappings. {II}.
\newblock {\em Acta Math.}, 191(1):55--107, 2003.

\bibitem[HPP23]{He:Pa:Pry:23}
Susanna Heikkil\"{a}, Pekka Pankka, and Eden Prywes.
\newblock Quasiregular curves of small distortion in product manifolds.
\newblock {\em J. Geom. Anal.}, 33(1):Paper No. 1, 44, 2023.

\bibitem[HS14]{Haj:Sch:14}
Piotr Haj{\l}asz and Armin Schikorra.
\newblock Lipschitz homotopy and density of {L}ipschitz mappings in {S}obolev
  spaces.
\newblock {\em Ann. Acad. Sci. Fenn. Math.}, 39(2):593--604, 2014.

\bibitem[IM93]{Iwa:Mar:93}
Tadeusz Iwaniec and Gaven Martin.
\newblock Quasiregular mappings in even dimensions.
\newblock {\em Acta Math.}, 170(1):29--81, 1993.

\bibitem[IM01]{Iw:Ma:01}
Tadeusz Iwaniec and Gaven Martin.
\newblock {\em Geometric function theory and non-linear analysis}.
\newblock Oxford Mathematical Monographs. The Clarendon Press, Oxford
  University Press, New York, 2001.

\bibitem[{Kan}21]{Kan:21}
Ilmari {Kangasniemi}.
\newblock {Notes on quasiregular maps between Riemannian manifolds}.
\newblock {\em arXiv e-prints}, page arXiv:2109.01638, September 2021.

\bibitem[KP19]{Kan:Pan:19}
Ilmari Kangasniemi and Pekka Pankka.
\newblock Uniform cohomological expansion of uniformly quasiregular mappings.
\newblock {\em Proc. Lond. Math. Soc. (3)}, 118(3):701--728, 2019.

\bibitem[KP22]{Kan:Pry:22}
Ilmari {Kangasniemi} and Eden {Prywes}.
\newblock {On the Moduli of Lipschitz Homology Classes}.
\newblock {\em arXiv e-prints}, page arXiv:2208.14517, August 2022.

\bibitem[Kö79]{Ko:79}
Gottfried Köthe.
\newblock {\em Topological vector spaces. {II}}, volume 237 of {\em Grundlehren
  der Mathematischen Wissenschaften [Fundamental Principles of Mathematical
  Sciences]}.
\newblock Springer-Verlag, New York-Berlin, 1979.

\bibitem[Lac74]{Lac:74}
H.~Elton Lacey.
\newblock {\em The isometric theory of classical {B}anach spaces}.
\newblock Die Grundlehren der mathematischen Wissenschaften, Band 208.
  Springer-Verlag, New York-Heidelberg, 1974.

\bibitem[Lee13]{Lee:21}
John~M. Lee.
\newblock {\em Introduction to smooth manifolds}, volume 218 of {\em Graduate
  Texts in Mathematics}.
\newblock Springer, New York, second edition, 2013.

\bibitem[LW16]{Lyt:Wen:16}
Alexander Lytchak and Stefan Wenger.
\newblock Regularity of harmonic discs in spaces with quadratic isoperimetric
  inequality.
\newblock {\em Calc. Var. Partial Differential Equations}, 55(4):Art. 98, 19,
  2016.

\bibitem[LW17]{Lyt:Wen:17:energyarea}
Alexander Lytchak and Stefan Wenger.
\newblock Energy and area minimizers in metric spaces.
\newblock {\em Adv. Calc. Var.}, 10(4):407--421, 2017.

\bibitem[LW18]{Lyt:Wen:18:CAT}
Alexander Lytchak and Stefan Wenger.
\newblock Isoperimetric characterization of upper curvature bounds.
\newblock {\em Acta Math.}, 221(1):159--202, 2018.

\bibitem[LW20]{Lyt:Wen:20}
Alexander Lytchak and Stefan Wenger.
\newblock Canonical parameterizations of metric disks.
\newblock {\em Duke Math. J.}, 169(4):761--797, 2020.

\bibitem[OP21]{Onn:Pan:21}
Jani Onninen and Pekka Pankka.
\newblock Quasiregular curves: {H}\"{o}lder continuity and higher
  integrability.
\newblock {\em Complex Anal. Synerg.}, 7(3):Paper No. 26, 9, 2021.

\bibitem[Pan08]{Pan:08}
Pekka Pankka.
\newblock Slow quasiregular mappings and universal coverings.
\newblock {\em Duke Math. J.}, 141(2):293--320, 2008.

\bibitem[Pan20]{Pan:20}
Pekka Pankka.
\newblock Quasiregular curves.
\newblock {\em Ann. Acad. Sci. Fenn. Math.}, 45(2):975--990, 2020.

\bibitem[PRW15]{Cam:Raj:Wen:15}
Camille Petit, Kai Rajala, and Stefan Wenger.
\newblock Wolfe's theorem for weakly differentiable cochains.
\newblock {\em J. Funct. Anal.}, 268(8):2261--2297, 2015.

\bibitem[PV20]{Pigo:Vero:20}
Stefano Pigola and Giona Veronelli.
\newblock The smooth {R}iemannian extension problem.
\newblock {\em Ann. Sc. Norm. Super. Pisa Cl. Sci. (5)}, 20(4):1507--1551,
  2020.

\bibitem[Raj21]{Raj:21}
Tapio Rajala.
\newblock Approximation by uniform domains in doubling quasiconvex metric
  spaces.
\newblock {\em Complex Anal. Synerg.}, 7(1):Paper No. 4, 5, 2021.

\bibitem[Res66a]{Res:66}
Yu.~G. Reshetnyak.
\newblock Bounds on moduli of continuity for certain mappings.
\newblock {\em Sib. Math. J.}, 7:879--886, 1966.

\bibitem[Res66b]{Res:66:isoperi}
Yu.~G. Reshetnyak.
\newblock Some geometrical properties of functions and mappings with
  generalized derivatives.
\newblock {\em Sib. Math. J.}, 7:704--732, 1966.

\bibitem[Res89]{Res:89}
Yu.~G. Reshetnyak.
\newblock {\em Space mappings with bounded distortion}, volume~73 of {\em
  Translations of Mathematical Monographs}.
\newblock American Mathematical Society, Providence, RI, 1989.
\newblock Translated from the Russian by H. H. McFaden.

\bibitem[Res97]{Res:97}
Yu.~G. Reshetnyak.
\newblock Sobolev classes of functions with values in a metric space.
\newblock {\em Sibirsk. Mat. Zh.}, 38(3):657--675, iii--iv, 1997.

\bibitem[Ric93]{Rick:93}
Seppo Rickman.
\newblock {\em Quasiregular mappings}, volume~26 of {\em Ergebnisse der
  Mathematik und ihrer Grenzgebiete (3) [Results in Mathematics and Related
  Areas (3)]}.
\newblock Springer-Verlag, Berlin, 1993.

\bibitem[RW13]{Raj:Wen:13}
Kai Rajala and Stefan Wenger.
\newblock An upper gradient approach to weakly differentiable cochains.
\newblock {\em J. Math. Pures Appl. (9)}, 100(6):868--906, 2013.

\bibitem[Sha00]{Sha:00}
Nageswari Shanmugalingam.
\newblock Newtonian spaces: an extension of {S}obolev spaces to metric measure
  spaces.
\newblock {\em Rev. Mat. Iberoamericana}, 16(2):243--279, 2000.

\bibitem[{Smi}11]{Smi:11}
Aaron~M. {Smith}.
\newblock {A theory of multiholomorphic maps}.
\newblock {\em arXiv e-prints}, page arXiv:1112.1471, December 2011.

\bibitem[Wen07]{Wen:07}
Stefan Wenger.
\newblock Flat convergence for integral currents in metric spaces.
\newblock {\em Calc. Var. Partial Differential Equations}, 28(2):139--160,
  2007.

\bibitem[Wil12]{Wil:12}
Marshall Williams.
\newblock Geometric and analytic quasiconformality in metric measure spaces.
\newblock {\em Proc. Amer. Math. Soc.}, 140(4):1251--1266, 2012.

\end{thebibliography}

\end{document}